\documentclass[12pt,a4paper]{amsart} 
\usepackage{amsmath,amssymb, amsthm}
\usepackage[T1]{fontenc}
\usepackage[utf8]{inputenc}
\usepackage[english]{babel}
\usepackage{tikz}
\usepackage{tkz-berge}
\usepackage{hyperref}
\usepackage{refcount}

\makeatletter
\newcommand\myitem[1][]{%
  \item[#1]\def\@currentlabel{#1}%
}
\makeatother

\newcommand{\abs}[1]{\ensuremath{\left|#1\right|}}

\newcommand{\halmaz}[1]{\ensuremath{\left\{\,#1\,\right\}}}
\newcommand{\halmazvonal}[2]{\ensuremath{\left\{\,#1\mid #2\,\right\}}}

\newcommand{\cpx}[1]{\ensuremath{\#_G\left( #1 \right)}}
\newcommand{\Green}{\mathcal}

\theoremstyle{plain}
\newtheorem{thm}{Theorem}

\newtheorem{lem}[thm]{Lemma}
\newtheorem{cor}[thm]{Corollary}
\theoremstyle{definition}
\newtheorem{dfn}[thm]{Definition}
\newtheorem{prob}{Problem}

\begin{document}
\title[Maximal subgroups and complexity of the flow semigroup]{The maximal subgroups and the complexity of the flow semigroup of finite (di)graphs}

\dedicatory{Dedicated to John Rhodes on the occasion of his 80th birthday.}

\author{G\'{a}bor Horv\'{a}th}
\address{Institute of Mathematics, 
University of Debrecen,
Pf.~400,
Debrecen,
4002,
Hungary}
\email{ghorvath@science.unideb.hu}

\author{Chrystopher L.~Nehaniv}
\address{Royal Society / Wolfson Foundation Biocomputation Research Laboratory,
Centre for Computer Science and Informatics Research, 
University of Hertfordshire,
College Lane, Hatfield,
Hertfordshire AL10 9AB, United Kingdom}
\email{C.L.Nehaniv@herts.ac.uk}

\author{K\'{a}roly Podoski}
\address{Alfr\'{e}d R\'{e}nyi Institute of Mathematics, 
13--15 Re\'{a}ltanoda utca, 
Budapest, 
1053,
Hungary}
\email{pcharles@renyi.mta.hu}

\thanks{
The research was partially supported by the European Council under the European Union's Seventh Framework Programme (FP7/2007-2013)/ERC under grant agreements no.~318202 and no.~617747,
by the MTA R\'{e}nyi Institute Lend\"{u}let Limits of Structures Research Group, 
the first author was partially supported by the Hungarian 
National Research, Development and
Innovation Office (NKFIH) grant no.~K109185
\ and grant no.~FK124814,
and the third author was funded by the National Research, Development and Innovation Office (NKFIH) Grant No. ERC\_HU\_15 118286.
}

\date{30 July 2017}

\subjclass[2010]{20M20, 05C20, 05C25, 20B30}

\keywords{Rhodes's conjecture, flow semigroup of digraphs, Krohn--Rhodes complexity, complete invariants for graphs, invariants for digraphs, permutation groups}

\begin{abstract}
The flow semigroup, introduced by John Rhodes, 
is an invariant for digraphs and a complete invariant for graphs. 
After collecting together previous partial results,
we refine and prove Rhodes's conjecture on the structure of the maximal groups in the flow semigroup 
for finite, antisymmetric, strongly connected digraphs. 

Building on this result, 
we investigate and fully describe the structure and actions of the maximal subgroups of the flow semigroup acting on all but $k$ points for all finite digraphs and graphs for all $k\geq 1$.  
A linear algorithm (in the number of edges) is presented to 
determine 
these so-called `defect $k$ groups' for any finite (di)graph.

Finally, 
we prove that the complexity of the flow semigroup of a 2-vertex connected (and strongly connected di)graph with $n$ vertices is $n-2$, 
completely confirming Rhodes's conjecture for such (di)graphs. 
\end{abstract}

\maketitle

\section{Introduction}\label{sec:intro}

John Rhodes in \cite{wildbook} introduced the {\em flow semigroup}, an invariant for graphs and digraphs
(that is, 
isomorphic flow semigroups correspond to isomorphic digraphs). 
In the case of graphs, this is a complete invariant determining the graph up to isomorphism.
The flow semigroup is the semigroup of transformations of the vertices generated by elementary collapsings corresponding to
the edges of the (di)graph.  
An elementary collapsing corresponding to the directed edge $uv$ is a map on the vertices moving $u$ to $v$ and acting as the identity on all other vertices. 
(See Section~\ref{sec:digraphs} for all the precise definitions.) 

 A maximal subgroup of this semigroup  for a finite (di)graph $D=(V_D, E_D)$ acts by permutations on all but $k$  of its vertices  ($1 \leq k \leq \abs{V_D}-1$)  and is called
a ``defect $k$ group''.  
The set of defect $k$ groups of a (di)graph is also an invariant. 
For each fixed $k$, they are all
isomorphic to each other in the case of (strongly) connected (di)graphs.   
Rhodes formulated a conjecture on the structure of these groups for 
strongly connected digraphs whose edge relation is anti-symmetric 
in \cite[Conjecture~6.51i (2)--(4)]{wildbook}.
We show that his conjecture was correct, and we prove it here in sharper form. 
Moreover, extending this result, 
we fully determine the defect $k$ groups for all finite graphs and digraphs.

Rhodes further conjectured \cite[Conjecture~6.51i (1)]{wildbook} 
that the Krohn--Rhodes complexity of the flow semigroup of a strongly connected, antisymmetric digraph $D$ on $n$ vertices is $n-2$. 
We confirm this conjecture 
when the digraph is 2-vertex connected, 
and bound the complexity in the remaining cases. 

The structure of the argument is as follows. 
First, 
a maximal group in the flow semigroup of a digraph $D$ is the direct product of maximal groups of the flow semigroups of its strongly connected components. 
Thus one needs only to consider strongly connected digraphs. 
It turns out, 
that if $D$ is a strongly connected digraph, 
then the defect $k$ group (up to isomorphism) does not depend on the choice of the vertices it acts on. 
Furthermore, 
for a strongly connected digraph, 
its flow semigroup is the same as the flow semigroup of the simple graph obtained by ``forgetting'' the direction of the edges. 
This is detailed in Section~\ref{sec:digraphs} and is based on \cite[p.~159--169]{wildbook}. 
Thus, 
one only needs to consider the defect $k$ groups of the flow semigroup for simple connected graphs. 

In Section~\ref{sec:prelim} we list some useful lemmas and determine the defect $k$ group of a cycle. 
In Section~\ref{sec:defect1} we prove that the defect 1 group of arbitrary simple connected graph is the direct product of the defect 1 groups of its 2-vertex connected components. 
The defect 1 group of an arbitrary 2-vertex connected graph $\Gamma$ has been determined by Wilson \cite{wilson}. 
He proved that the defect 1 group is either $A_{n-1}$ or $S_{n-1}$, 
unless $\Gamma$ is a cycle or the exceptional graph displayed in Figure~\ref{fig:exceptionalgraph}. 

\begin{figure}
\begin{center}
\begin{tikzpicture}[-,>=stealth',shorten >=1pt,auto,node distance=2cm, thin, 
 main node/.style={circle,draw},
 rectangle node/.style={rectangle,draw},
 empty node/.style={}]

  \node[rectangle node] (v)  {$v$};
  \node[main node] (4) [above left of=v] {$4$};
  \node[main node] (3) [left of=4] {$3$};
  \node[main node] (2) [below left of=3] {$2$};
  \node[main node] (5) [below right of=2] {$5$};
  \node[main node] (6) [right of=5] {$6$};
  \node[main node] (1) [left of=v] {$1$};

  \path[every node/.style={font=\sffamily\small}]
    (v) edge node {} (1)
    (1) edge node {} (2)
    (2) edge node {} (3)
    (3) edge node {} (4)
    (4) edge node {} (v)
    (5) edge node {} (2)
    (5) edge node {} (6)
    (6) edge node {} (v);
\end{tikzpicture}
\caption{Exceptional graph}\label{fig:exceptionalgraph}
\end{center}
\end{figure}
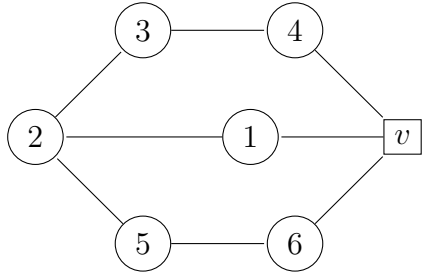

In particular, 
Rhodes's conjecture 
(as phrased for strongly connected, 
antisymmetric digraphs in \cite[Conjecture~6.51i~(2)]{wildbook}) 
about the defect 1 group holds, 
and more generally: 
the defect 1 group of the flow semigroup of a simple connected graph is indeed the product of cyclic, 
alternating and symmetric groups of various orders. 
A straightforward linear algorithm is given
to determine the direct components of the defect 1 group of an arbitrary connected graph
(see Section~\ref{sec:algorithm}). 

In Section~\ref{sec:defectk} we determine the defect $k$ groups ($k \geq 2$) of arbitrary graphs by
considering the so-called maximal $k$-subgraphs
(maximal subgraphs for which the defect $k$ group is the full symmetric group) 
and prove that the defect $k$ group of a graph is the direct product of the defect $k$ groups of the maximal $k$-subgraphs (i.e.\ of full symmetric groups).
In Section~\ref{sec:algorithm} we provide a linear algorithm (in the number of edges of $\Gamma$) to determine the maximal $k$-subgraphs of an arbitrary connected graph. 
Finally, 
in Section~\ref{sec:cpx} we confirm \cite[Conjecture~6.51i (1)]{wildbook} about the Krohn--Rhodes complexity of digraphs when the digraph is 2-vertex connected, 
and we prove some bounds on the complexity of the flow semigroup in the remaining cases. 
(See Section~\ref{sec:cpx} for the definition of Krohn--Rhodes complexity.)

We have collected all these results into the following main theorem. 

\begin{thm}\label{thm:main}\ 
\begin{enumerate}
\item\label{it:digraph}
Let $D$ be a digraph, 
then every maximal subgroup of $S_D$ is (isomorphic to) the direct product of maximal subgroups of $S_{D_i}$, 
where the $D_i$ are the strongly connected components of $D$. 
\item\label{it:stronglyG_k}
Let $D$ be a strongly connected digraph. 
Let $V_k, V_k' \subseteq D$ be subsets of nodes such that $\abs{V_k} = \abs{V_k'}=k$. 
Let $G_{k, V_k}, G_{k, V_k'}$ be the defect $k$ groups acting on $V \setminus V_k$ and $V \setminus V_k'$, respectively. 
Then $G_{k, V_k} \simeq G_{k, V_k'}$
as permutation groups. 
\myitem[(\getrefnumber{it:stronglyG_k}\textsuperscript{r})]\label{it:stronglyS_D}
Let $D$ be a strongly connected digraph, 
and $\Gamma_D$ be the graph obtained from $D$  by forgetting the direction of the edges in $D$.
Then $S_D = S_{\Gamma_D}$. 
\item\label{it:graph}
Let $\Gamma$ be a simple connected graph of $n$ vertices, 
and let $\Gamma_1, \dots , \Gamma_m$ be its 2-vertex connected components. 
Then the defect 1 group of $\Gamma$ is the direct product of the defect 1 groups of $\Gamma_i$ ($1\leq i\leq m$). 
\item\label{it:2-vertex}
Let $\Gamma$ be a 2-vertex connected simple graph with $n \geq 2$ vertices. 
Then the defect 1 group of $\Gamma$ is isomorphic (as a permutation group) to
\begin{enumerate}
\item\label{it:defect1cycle}
the cyclic group $Z_{n-1}$ if $\Gamma$ is a cycle; 
\item
$S_5 \simeq PGL_2(5)$ acting sharply 3-transitively on 6 points, 
if $\Gamma$ is the exceptional graph (see Figure~\ref{fig:exceptionalgraph});
\item 
$S_{n-1}$ or $A_{n-1}$, otherwise, 
where the defect 1 group is $A_{n-1}$ if and only if $\Gamma$ is bipartite.
\end{enumerate}
\myitem[(\getrefnumber{it:2-vertex}\textsuperscript{c})]\label{it:compl2vertex}
Let $\Gamma$ be a 2-vertex connected simple graph with $n \geq 2$ vertices. 
Then the complexity of $S_\Gamma$ is $\cpx{S_{\Gamma}}=n-2$. 
\myitem[(\getrefnumber{it:2-vertex}\textsuperscript{cc})]\label{it:compl2edge}
Let $\Gamma$ be a 2-edge connected simple graph with $n \geq 2$ vertices. 
Then for the complexity of $S_\Gamma$ we have $n-3 \leq \cpx{S_{\Gamma}} \leq n-2$. 
\item\label{it:defectk}
Let $k \geq 2$, 
$\Gamma$ be a simple connected graph of $n$ vertices, 
$n > k$. 
\begin{enumerate}
\item\label{it:defectkcycle}
If $\Gamma$ is a cycle, 
then its defect $k$ group is the cyclic group $Z_{n-k}$. 
\item\label{it:defectkmain}
Otherwise, let
$\Gamma_1, \dots , \Gamma_m$ be the maximal $k$-subgraphs of $\Gamma$, 
and let $\Gamma_i$ have $n_i$ vertices. 
Then the defect $k$ group of $\Gamma$ is the direct product of the defect $k$ groups of $\Gamma_i$ ($1 \leq i\leq m$), 
thus it is isomorphic (as a permutation group) to 
\[
S_{n_1-k} \times \dots \times S_{n_m-k}. 
\]
\end{enumerate}
\end{enumerate}
\end{thm}

Our main contribution to Theorem~\ref{thm:main} are items~(\ref{it:graph}), \ref{it:compl2vertex}, \ref{it:compl2edge} and (\ref{it:defectk}). 
Items~(\ref{it:digraph}), (\ref{it:stronglyG_k})~and~\ref{it:stronglyS_D} (among some basic definitions and notations) are detailed in Section~\ref{sec:digraphs} and are based on \cite[p.~159--169]{wildbook}. 
In Section~\ref{sec:prelim} we list some useful lemmas and determine the defect $k$ group of a cycle. 
Item~(\ref{it:graph}) is proved in Section~\ref{sec:defect1}, 
while item~(\ref{it:2-vertex}) has already been proved by Wilson \cite{wilson}. 
Then in Section~\ref{sec:defectk} we prove item~(\ref{it:defectk}). 
In Section~\ref{sec:algorithm} we provide a linear algorithm (in the number of edges of $\Gamma$) to determine the maximal $k$-subgraphs of an arbitrary connected graph to help putting item~(\ref{it:defectk}) more into context. 
Finally, 
items~\ref{it:compl2vertex}~and~\ref{it:compl2edge} are proved in Section~\ref{sec:cpx}. 

East, Gadouleau and Mitchell \cite{mitchell} are currently looking into other properties of flow semigroups. 
In particular, 
they provide a linear algorithm (in the number of vertices of a digraph) for whether or not the flow semigroup contains a cycle of length $m$ for a fixed positive integer $m$. 
Furthermore, 
they classify all those digraphs whose flow semigroups have any of the following properties: 
inverse, completely regular, commutative, simple, 0-simple, 
a semilattice, a rectangular band, 
congruence-free, 
is $\Green{K}$-trivial or $\Green{K}$-universal, where $\Green{K}$ is any of Green's $\Green{H}$-, $\Green{L}$-, $\Green{R}$-, or $\Green{J}$-relation, 
and when the flow-semigroup has a left, right, or two-sided zero. 

Rhodes's original conjecture \cite[Conjecture~6.51i]{wildbook} is about strongly connected, 
antisymmetric digraphs. 
By \cite{Robbins1939} 
a strongly connected antisymmetric digraph becomes a 2-edge connected graph after forgetting the directions. 
Therefore Theorem~\ref{thm:main} almost completely settles Rhodes's conjecture \cite[Conjecture 6.51i]{wildbook}. 
To completely settle the last remaining part of Rhodes's conjecture \cite[Conjecture 6.51i (1)]{wildbook}, 
one should find the complexity of the flow semigroups for the rest of the 2-edge connected graphs.

\begin{prob}
Determine the complexity of $S_{\Gamma}$ for a 2-edge connected graph $\Gamma$ which
is not 2-vertex connected.
\end{prob}

The smallest such graph is the ``bowtie'' graph:

\begin{prob}
Let $\Gamma$ be the  graph with vertex set $\halmaz{u,v,w,x,y}$ and edge set 
$\halmaz{uv,vw,wu,wx,xy,yw}$.
Determine the complexity of $S_{\Gamma}$. 
\end{prob}

Ultimately, the goal is the determine the complexity for all flow semigroups.
\begin{prob}
Determine the complexity of $S_{\Gamma}$ for an arbitrary finite graph (or digraph) $\Gamma$.
\end{prob}

\section{Flow semigroup of digraphs}\label{sec:digraphs}

For notions in graph theory we refer to \cite{GraphTheory, HandbookOfCombinatorics}, 
in group theory to \cite{Robinson}
in permutation groups to \cite{Cameron, DixonMortimer}, 
in semigroup theory to \cite{CliffordPreston1, CliffordPreston2}. 

A {\em semigroup} is a set with a binary associative multiplication. 
A {\em transformation} on a set $X$ is  a function $s \colon X\rightarrow X$.  
It {\em operates} (or {\em acts}) {\em on} $X$ by mapping each $x\in X$ to some $x\cdot s \in X$.   Here we write $x\cdot s$ or $xs$ for transformation $s$ applied to $x\in X$.  
A {\em transformation semigroup} $S$  is a set of transformations $s \in S$ on some set $X$ such that $S$ is closed under (associative)
 function composition.   Also, $S$ itself is then said to operate or {\em to act on} the set $X$. Note that in this paper functions act on the right,  therefore transformations are multiplied from left to right.  Denoting by $ss'$ the transformation of $X$ obtained by first applying $s$ and then $s'$, 
we have $x \cdot ss' = (x \cdot s) \cdot s'$.
If a semigroup element $s$ acts on a set $X$, 
and for some $Y \supseteq X$ the action of $s$ is not defined on $Y \setminus X$, 
then we may consider $s$ acting on $Y$, 
as well, 
with the identity action on $Y \setminus X$.

 A {\em permutation group} is a nonempty transformation semigroup $G$ that contains only permutations and such that
that if  $g \in G$ then the inverse permutation $g^{-1}$ is also in $G$.   Furthermore, 
for a set $Y \subseteq  X$ and a transformation $s$ on $X$ define
\[
Ys = \halmazvonal{ys}{y \in Y}. 
\]
A {\em subgroup} $G$ of a transformation semigroup $S$ is a subset of $S$ whose transformations satisfy the (abstract) group axioms. It is not hard to show that
 if $S$ is a transformation semigroup acting on $X$, then
$G$ contains a (unique) idempotent $e^2=e$ (which does not generally act as the identity map on $X$), and
furthermore distinct elements  of 
$G$ when restricted to $X e$ are distinct,  permute $Xe$, and comprise a permutation group acting on $Xe$ (see \cite[p.~49]{wildbook}).

A {\em digraph} $(V, E)$ is a set of \emph{nodes} (or \emph{vertices}) $V$, and a binary relation $E \subseteq V\times V$.
An element  $e=(u,v) \in E$ is called a {\em directed edge} from node $u$ to node $v$, and also denoted $uv$. A {\em loop-edge} is an edge from a vertex to itself. 
A {\em graph} $(V,E)$ is a set of nodes $V$ and a symmetric binary relation $E \subseteq V \times V$.
If $(u,v) \in E$, 
then $uv$ is called an (undirected) edge. Such a graph is called {\em simple} if it has no loop-edges.
\emph{In this paper we consider only digraphs without loop-edges and simple graphs}.
A \emph{walk} is a sequence of vertices $\left( v_1, \dots, v_n \right)$ such that $v_iv_{i+1}$ is a (directed) edge  for all $1\leq i\leq n-1$.
By \emph{cycle} we will mean a simple cycle,
that is a closed walk with no repetition of 
vertices except for the starting and ending vertex.
A \emph{path} is a walk with no repetition of 
vertices. 
A (di)graph $\Gamma = (V, E)$ is (strongly) connected if there is a path from $u$ to $v$ for all distinct $u, v \in V$. 
By \emph{subgraph} $\Gamma' = (V', E') \subseteq \Gamma$ we mean a graph for which $V' \subseteq V$,
$E' \subseteq E$.
If $\Gamma'$ is an \emph{induced subgraph},
that is $E'$ consists of all edges from $E$ with both endpoints in $V'$,
then we explicitly indicate it.
A strongly connected component of a digraph $\Gamma$ is a maximal strongly connected subgraph of $\Gamma$.

\bigskip

For a digraph $D = (V_D, E_D)$ without any loop-edges, the \emph{flow semigroup} $S = S_D$ is the semigroup of transformations acting on $V_D$ defined by 
\[
S = S_D = \left<e_{uv} \mid uv \in E_D \right>,
\]
where $e_{uv}$ is the \emph{elementary collapsing} corresponding to the directed edge $uv \in E_D$, 
that is, 
for every $x \in V_D$ we have 
\[
x\cdot e_{uv} = xe_{uv} = 
\begin{cases}
v, & \text{if } x=u, \\
x, & \text{otherwise.}
\end{cases}
\]
Thus, the flow semigroup of a (di)graph $D$ is generated by idempotents (elementary collapsings) corresponding to the edges of $D$. 
The flow semigroup $S_D$ is also called the {\em Rhodes semigroup of the (di)graph}.

\bigskip

A maximal subgroup of $S_D$ is a subgroup that is not properly contained in any other subgroup of $S_D$.
In order to determine the maximal subgroups of $S_D$,
one can make several reductions by \cite[Proposition~6.51f]{wildbook}.
First, 
one only needs to consider the maximal subgroups of $S_{D_i}$ for the strongly connected components $D_i$ of $D$. 
Strongly connected components are maximal induced subgraphs such that any vertex can be reached from any other vertex by a directed path.

\begin{lem}[{\cite[Proposition~6.51f (1)]{wildbook}}]
Let $D$ be a digraph, 
then every maximal subgroup of $S_D$ is (isomorphic to) the direct product of maximal subgroups of $S_{D_i}$, 
where the $D_i$ are the strongly connected components of $D$. 
\end{lem}

This is (\ref{it:digraph}) of Theorem~\ref{thm:main}. 
An element $s \in S$ is \emph{of defect $k$} if $\abs{V_D s} = \abs{V_D}-k$. 
Let $V_k=\halmaz{ v_1,v_2,\dots , v_k} \subseteq V_D$.
The \textit{defect $k$ group} $G_{k,V_k}$ associated to $V_k$ (called the \emph{defect set})
is generated by all elements of $S$ restricted to $V_D\setminus V_k$ which permute the elements of $V_D\setminus V_k$
and move elements of $V_k$ to elements of $V_D\setminus V_k$:
\[
G_{k, V_k}=\left< s\restriction_{V_D\setminus V_k}\, : \, s\in S, (V_D\setminus V_k)s=V_D\setminus V_k, V_k s \subseteq V_D\setminus V_k \right>,
\]
where $s\restriction_{V_D\setminus V_k}$ denotes the restriction of the transformation $s$ onto the set $V_D\setminus V_k$. 
Now, $G_{k, V_k}$ is a permutation group acting on $V_D\setminus V_k$. 
For this reason $V_D\setminus V_k$ is called the \emph{permutation set} of $G_{k, V_k}$, 
and the elements of $G_{k, V_k}$ are sometimes called \emph{defect $k$ permutations}.
Furthermore, 
if the defect set contains only one vertex $v$,
then by abuse of notation we write \emph{defect $v$} or \emph{defect point $v$} instead of defect $\halmaz{v}$. 
In general, the defect $k$ group $G_{k, V_k}$ can depend on the choice of $V_k$. 
However, 
by \cite[Proposition~6.51f (2)]{wildbook} it turns out that if the graph is strongly connected
then the defect $k$ group $G_k$ is unique up to isomorphism.

\begin{lem}[{\cite[Proposition~6.51f (2)]{wildbook}}]\label{lem:stronglyconnected}
Let $D$ be a strongly connected digraph. 
Let $V_k, V_k' \subseteq V_D$ be subsets of nodes such that $\abs{V_k} = \abs{V_k'}=k$. 
Then the action of $G_{k, V_k}$ on $V_D \setminus V_k$ is equivalent to that of $G_{k, V_k'}$ on $V_D \setminus V_k'$.
That is, 
$G_{k, V_k} \simeq G_{k, V_k'}$
as permutation groups.
\end{lem}

This is (\ref{it:stronglyG_k}) of Theorem~\ref{thm:main}.
By Lemma~\ref{lem:stronglyconnected}, 
we may write $G_k$ instead of $G_{k, V_k}$ without any loss of generality.
Furthermore, 
the case of strongly connected graphs can be reduced to the case of simple graphs. 
Let $\Gamma=(V, E)$ be a simple (undirected) graph, 
we define $S_{\Gamma}$ by considering $\Gamma$ as a directed graph where every edge is directed both ways. 
Namely, 
let $D_{\Gamma}=(V,E_D)$ be the directed graph on vertices $V$ such that both $uv \in E_D$ and $vu \in E_D$ if and only if the undirected edge $uv \in E$.  
Then let $S_{\Gamma} = S_{D_\Gamma}$. 

Furthermore, 
for every digraph $D = (V_D, E_D)$, 
one can associate an undirected graph $\Gamma$ by ``forgetting'' the direction of edges in $D$. 
Precisely, 
let $\Gamma_D=(V_D,E)$ be the undirected graph such that $uv \in E$ if and only if $uv \in E_D$ or $vu \in E_D$. 
The following lemma due to Nehaniv and Rhodes shows that if a digraph $D$ is strongly connected then the semigroup
$S_D$ corresponding to $D$ and the semigroup $S_{\Gamma_D}$ corresponding to the simple graph $\Gamma_D$ are the same. 
Moreover, Lemma~\ref{lem:reverseedge} immediately implies that the transformation semigroup $S_D$ is an
invariant for digraphs and a complete invariant for (simple) graphs: That is, isomorphic digraphs have the isomorphic flow semigroups, and graphs are isomorphic if and only if their flow semigroups are isomorphic as transfromation semigroups.

\begin{lem}[{\cite[Lemma~6.51b]{wildbook}}]\label{lem:reverseedge}
Let $D$ be an arbitrary digraph. 
Then
\[
e_{ab} \in S_D \Longleftrightarrow
\begin{cases} 
a \to b \text{ is an edge in $D$, or } \\
b \to a \text{ is an edge in a directed cycle in $D$.} 
\end{cases} 
\]
In particular, 
if $D$ is strongly connected then $S_D = S_{\Gamma_D}$. 
\end{lem}

\begin{proof}
Let $b \to a \to u_1 \to \dots \to  u_{n-1} \to b$ be a directed cycle in $D$.
Then an easy calculation shows that
\[
e_{ab} = \left(e_{ba}e_{u_{n-1}b}e_{u_{n-2}u_{n-1}} \dots e_{u_{1}u_2} e_{a u_1} \right)^{n}. 
\]
For the other direction, 
assume $e_{ab} = e_{uv} s$ for some $s \in S_D$.
Then $e_{uv}s$ moves $u$ and $v$ to the same vertex, 
while $e_{ab}$ moves only $a$ and $b$ to the same vertex. 
Thus $\halmaz{a,b} = \halmaz{u,v}$. 
\end{proof}

This is \ref{it:stronglyS_D} of Theorem~\ref{thm:main}. 
Therefore, in the following
we only consider simple, connected, undirected graphs $\Gamma = (V, E)$, 
that is no self-loops or multiple edges are allowed.
Furthermore, 
$\Gamma$ is 2-edge connected if removing any edge does not disconnect $\Gamma$. 
Rhodes's conjecture \cite[Conjecture~6.51i (2)--(4)]{wildbook} is about strongly connected, 
antisymmetric digraphs. 
Note 
that by \cite{Robbins1939} 
a strongly connected antisymmetric digraph becomes a 2-edge connected graph after forgetting the directions. 

Let us fix some notation.
The letters $k$, $l$, $m$ and $n$ will denote nonnegative integers.
The number of vertices of $\Gamma$ is usually denoted by $n$,
while $k$ will denote the size of the defect set.
Usually we denote the defect $k$ group of a graph $\Gamma$ by $G_k$ or $G_{\Gamma}$,
depending on the context.
We try to heed the convention of using $u$, $v$, $w$, $x$, $y$ as vertices of graphs,
$V$ as the set of vertices,
$E$ as the set of edges.
Furthermore,
the flow semigroup is mostly denoted by $S$,
its elements are denoted by $s$, $t$, $g$, $h$, $p$, $q$.
The cyclic group of $m$ elements is denoted by $Z_m$. 

We will need the notion of an open ear, 
and open ear decomposition. 

\begin{dfn}
Let $\Gamma$ be an arbitrary graph, 
and let $\Gamma'$ be a proper subgraph of $\Gamma$.
A path $(u, c_1, \dots, c_m, v)$ is called a \emph{$\Gamma'$-ear} (or \emph{open ear}) with respect to $\Gamma$, 
if $u, v\in\Gamma'$, $u\neq v$, 
and either $m=0$ and the edge $uv \notin \Gamma'$, 
or $c_1, \dots, c_m \in \Gamma \setminus \Gamma'$.
An \emph{open ear decomposition} of a graph is a partition of its set of edges into a sequence of subsets, 
such that the first element of the sequence is a cycle, 
and all other elements of the sequence are open ears of the union of the previous subsets in the sequence.
\end{dfn}

A connected graph $\Gamma$ with at least $k$ vertices is \emph{$k$-vertex connected}
if removing any $k-1$ vertices does not disconnect $\Gamma$. 
By~\cite{Whitney1932} a graph is 2-vertex connected if and only if it is a single edge or it has an open ear decomposition.

\section{Preliminaries}\label{sec:prelim}

Let $\Gamma = (V, E)$ be a simple, connected (undirected) graph, 
and for every $1\leq k\leq \abs{V}-1$, let $G_k$ denote its defect $k$ group for some $V_k \subseteq V$, 
$\abs{V_k} = k$. 
Let $S = S_{\Gamma}$ be the flow semigroup of $\Gamma$. 
The following is immediate. 

\begin{lem}[{\cite[Fact~6.51c]{wildbook}}]\label{lem:sTisdefectk}
Let $s\in S$ be of defect $k$. 
If $se_{uv}$ is of defect $k$, as well, 
then $u \notin Vs$ or $v \notin Vs$. 
\end{lem}

Furthermore, 
it is not too hard to see that every defect 1 permutation arises from the permutations generated by cycles (in the graph) containing the defect point. 

\begin{lem}[{\cite[Proposition~6.51e]{wildbook}}]\label{lem:cyclepermutation}
Let $\Gamma$ be a connected graph, 
and let $G_{1}$ denote its defect 1 group, 
such that the defect point is $v \in V$. 
Then 
\[
G_{1} = \left< \left( u_1, \dots , u_k \right) \text{ as permutation} \mid (u_1, \dots , u_k, v) \text{ is a cycle in } \Gamma \right>. 
\]
\end{lem}

These yield that the defect $k$ group of the $n$-cycle graph is cyclic, 
proving items~(\ref{it:defect1cycle})~and~(\ref{it:defectkcycle}) of Theorem~\ref{thm:main}:

\begin{lem}\label{lem:cycle}
The defect $k$ group of the $n$-cycle is isomorphic to $Z_{n-k}$.
\end{lem}

\begin{proof}
Let $x_1,x_2,\dots x_{n}$ be the consecutive elements of  the cycle $\Gamma=(V, E)$.
If $s\in S$ is an element of defect $k$ 
then by Lemma~\ref{lem:sTisdefectk} we have that 
$se_{x_ix_{i+1}}$ is of defect $k$ if and only if $x_i\notin Vs$ or $x_{i+1}\notin Vs$.
This means that if $u_1,u_2,\dots u_{n-k}$ are the consecutive elements of $Vs$ in the cycle and 
$se_{x_ix_{i+1}}$ is of defect $k$, as well, then  
\[
u_1 e_{x_ix_{i+1}}, u_2 e_{x_ix_{i+1}}, \dots, u_{n-k}e_{x_ix_{i+1}}
\]
are the consecutive elements of $Vse_{x_ix_{i+1}}$.  
Thus the cyclic ordering of these elements cannot be changed. 
Hence $G_k$ is isomorphic to a subgroup of $Z_{n-k}$.

Now, assume that $v_1, v_2,\dots v_k, u_1, u_2,\dots u_{n-k}$ are the consecutive elements of $\Gamma$, 
and the defect set is $V_k = \halmaz{v_1, \dots , v_k}$.
Let 
\begin{align*}
s_1 &=e_{v_1v_2}\dots  e_{v_jv_{j+1}}\dots e_{v_{k-1}v_k}, \\
s_2 &= e_{u_{n-k}v_k}e_{u_{n-k-1}u_{n-k}}\dots e_{u_{j-1}u_j}\dots e_{u_1u_2}e_{v_{k}u_1}, \\
s &=s_1s_2. 
\end{align*}
It easy to check that 
\[
v_is=u_1, \ \ u_1s=u_2,\dots , u_j s=u_{j+1}, \dots,  u_{n-k} s=u_1.
\]
Therefore  $s, s^2,\dots, s^{n-k}$ are distinct elements of $G_k$, 
hence $G_k\simeq Z_{n-k}$.
\end{proof}

\section{Defect 1 groups}\label{sec:defect1}

In this Section we prove item~(\ref{it:graph})~of~Theorem~\ref{thm:main}, 
which states that the defect 1 group of a simple connected graph is the direct product of the defect 1 groups of its 2-vertex connected components. 
This follows by induction on the number of 2-vertex connected components from 
Lemma~\ref{lem:defect1directproduct}.
The case where $\Gamma$ is 2-vertex connected (that is item~(\ref{it:2-vertex}) of Theorem~\ref{thm:main}) is covered by \cite[Theorem~2]{wilson}.

\begin{lem}\label{lem:defect1directproduct}
Let $\Gamma_1$ and $\Gamma_2$ be connected induced subgraphs of $\Gamma$
such that $\Gamma_1\cap\Gamma_2 = \halmaz{v}$, 
where there are no edges in $\Gamma$ between $\Gamma_1 \setminus \halmaz{v}$ and $\Gamma_2 \setminus \halmaz{v}$. 
Then the defect $1$ group of $\Gamma_1\cup\Gamma_2$ is the direct product of 
the defect $1$ groups of $\Gamma_1$ and $\Gamma_2$.
\end{lem}

\begin{proof}
Let $G_{\Gamma_i}$ denote the defect 1 group of $\Gamma_i$, 
where the defect point is $v$. 
By Lemma~\ref{lem:cyclepermutation}, 
$G_{\Gamma}$ is generated by cyclic permutations corresponding to cycles through $v$ in $\Gamma$. 
Now, 
$\Gamma_1 \cap \Gamma_2 = \halmaz{v}$, 
and every path between a node from $\Gamma_1$ and a node from $\Gamma_2$ must go through $v$, 
hence every cycle in $\Gamma$ is either in $\Gamma_1$ or in $\Gamma_2$. 
Let $c_i^{(1)}, \dots , c_i^{(m_i)}$ be the permutations corresponding to the cycles in $\Gamma_i$ ($i =1, 2$).
Since these cycles do not involve $v$ by Lemma~\ref{lem:cyclepermutation},
we have $c_1^{(j_1)}c_2^{(j_2)} = c_2^{(j_2)}c_1^{(j_1)}$ for all $1\leq j_i \leq m_i$, $i=1, 2$, 
thus
\begin{multline*}
G_{\Gamma} = \left< c_1^{(1)}, \dots , c_1^{(m_1)}, c_2^{(1)}, \dots , c_2^{(m_2)} \right> \\
=\left< c_1^{(1)}, \dots , c_1^{(m_1)} \right> \times \left< c_2^{(1)}, \dots , c_2^{(m_2)} \right> = G_{\Gamma_1} \times G_{\Gamma_2}. 
\end{multline*}
\end{proof}

\section{Defect $k$ groups}\label{sec:defectk}

We prove item~(\ref{it:defectkmain}) of Theorem~\ref{thm:main} in this Section. 
In the following we assume $k \geq 2$, 
and every graph $\Gamma$ is assumed to be simple connected. 
We start with some simple observations.

\begin{lem}\label{lem:simple}
Let $\Gamma$ be a connected graph,
and let $\Gamma'$ be a connected subgraph of $\Gamma$.
If $\Gamma'$ has at least $k+1$ vertices, 
then the defect $k$ group of $\Gamma$ contains a subgroup isomorphic (as a permutation group) to the defect $k$ group of $\Gamma'$. 
Furthermore, 
if $\Gamma \setminus \Gamma'$ contains at least one vertex, 
and $\Gamma'$ has at least $k$ vertices, 
then the defect $k$ group of $\Gamma$ contains a subgroup isomorphic 
(as a permutation group) to the defect $k-1$ group of $\Gamma'$. 
\end{lem}

\begin{proof}
Let $\Gamma = \left(V, E \right)$, 
$\Gamma' = \left(V', E' \right)$. 
First, 
assume $\abs{V'} \geq k+1$, 
and let $V_k = \halmaz{v_1, \dots, v_k} \subseteq V'$. 
Let $G_{k, V_k}$ and $G'_{k, V_k}$ be the defect $k$-groups of $\Gamma$ and $\Gamma'$.
Let $g \in G'_{k, V_{k}}$ be arbitrary. 
Then there exists $s \in S_{\Gamma'}$ with defect set $V_{k}$ such that $s \restriction_{V' \setminus V_{k}} = g$. 
Now, $E' \subseteq E$, 
hence every elementary collapsing of $\Gamma'$ is an elementary collapsing of $\Gamma$, as well, 
Thus $s \in S_{\Gamma}$, 
and $s$ acts as the identity on $V \setminus V'$.
Furthermore, 
if $s' \in S_{\Gamma'}$ is another element with defect set $V_{k}$ such that $s' \restriction_{V' \setminus V_{k}} = g = s \restriction_{V' \setminus V_{k}}$, 
then $s' \in S_{\Gamma}$ with
$s' \restriction_{V \setminus V_k} = s \restriction_{V \setminus V_k}$. 
Thus $\varphi \colon G'_{k, V_{k}} \to G_{k, V_{k}}$, 
$\varphi(g) = s \restriction_{V \setminus V_k}$ is a well defined injective homomorphism of permutation groups. 

Second, 
assume $\abs{V'} \geq k$, 
and let $V_{k-1} = \halmaz{v_1, \dots, v_{k-1}} \subseteq V'$. 
Let $v \in V \setminus V'$, 
and let $V_k = V_{k-1} \cup \halmaz{v}$. 
Let $u$ be a neighbor of $v$ and let $e = e_{vu}$. 
Let $G_{k, V_k}$ be the defect $k$-group of $\Gamma$ and let $G'_{k-1, V_{k-1}}$ be the defect $(k-1)$-group of $\Gamma'$.
Let $g \in G'_{k-1, V_{k-1}}$ be arbitrary. 
Then there exists $s \in S_{\Gamma'}$ with defect set $V_{k-1}$ such that $s \restriction_{V' \setminus V_{k-1}} = g$. 
Now, 
$es \in S_{\Gamma}$ has defect set $V_k$, 
and $es \restriction_{V \setminus V_k}$ acts as $g$ on $V' \setminus V_{k-1}$, 
and acts as the identity on $V \setminus \left( V' \cup \halmaz{v} \right)$. 
Furthermore, 
if $s' \in S_{\Gamma'}$ is another element with defect set $V_{k-1}$ such that $s' \restriction_{V' \setminus V_{k-1}} = g = s \restriction_{V' \setminus V_{k-1}}$, 
then $es \restriction_{V \setminus V_k} = es' \restriction_{V \setminus V_k}$. 
As $g \in G'_{k-1, V_{k-1}}$ was arbitrary, 
we have that 
$\varphi \colon G'_{k-1, V_{k-1}} \to G_{k, V_{k}}$, 
$\varphi(g) = es \restriction_{V \setminus V_k}$ is a well defined injective homomorphism of permutation groups. 
\end{proof}

\begin{lem}\label{lem:sgraph2}
Let $1 \leq m\leq l < k \leq n-2$, 
and assume $\Gamma$ contains the following subgraph: 
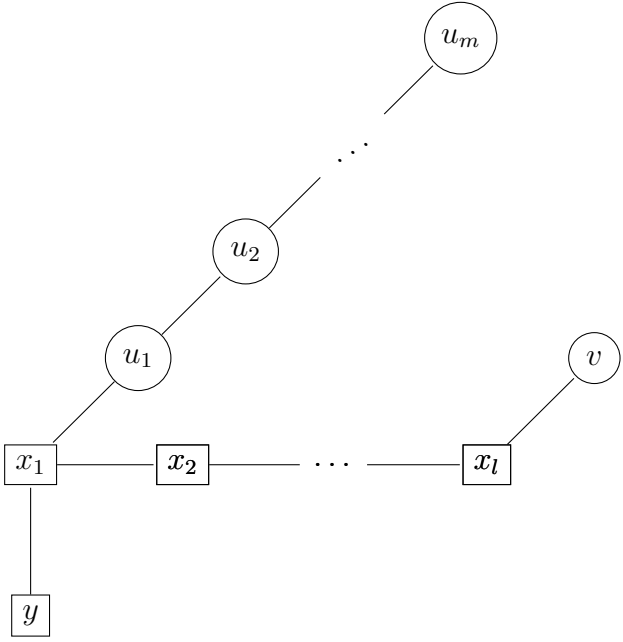
\begin{center}
\begin{tikzpicture}[-,>=stealth',shorten >=1pt,auto,node distance=2cm, thin, 
 main node/.style={circle,draw},
 rectangle node/.style={rectangle,draw},
 empty node/.style={}]

  \node[rectangle node] (x1) {$x_1$};
  \node[rectangle node] (y) [below of=x1] {$y$};
  \node[rectangle node] (x2) [right of=x1] {$x_2$};
  \node[empty node] (dots) [right of=x2] {$\dots$};
  \node[rectangle node] (xl) [right of=dots] {$x_{l}$};
  \node[main node] (v) [above right of=xl] {$v$};
  \node[main node] (u_1) [ above right  of=x1] {$u_1$};
  \node[rectangle node] (x2) [right of=x1] {$x_2$};
  \node[empty node] (dots) [right of=x2] {$\dots$};
  \node[rectangle node] (xl) [right of=dots] {$x_{l}$}; 
  \node[main node] (u_2) [above right of=u_1] {$u_2$};
  \node[empty node] (dots_1) [above right of=u_2] {\reflectbox{$\ddots$}};
  \node[main node] (u_m) [above right of=dots_1] {$u_m$};

  \path[every node/.style={font=\sffamily\small}]
    (u_1) edge node {} (x1)
    (x1) edge node {} (x2)
    (x2) edge node {} (dots)
    (dots) edge node {} (xl)
    (xl) edge node {} (v)
    (y) edge node {} (x1)
   (u_1) edge node {} (u_2)
   (u_2) edge node {} (dots_1)
   (dots_1) edge node {} (u_m);
 
\end{tikzpicture}
\end{center}
If $V_k$ is a set of nodes of size $k$ such that $y, x_1, \dots, x_l \in V_k$,
and $v, u_i \notin V_k$ for some $1\leq i\leq m$, 
then the defect $k$ group $G_{k, V_k}$ contains the transposition $(u_i, v)$. 
\end{lem}

\begin{proof}
Let 
\[
r = \begin{cases}
s s_1 e_{yx_1} e_{x_1u_1}, & \text{ if } i =1, \\
s s_1 \dots s_i p t t_{i-1} \dots t_1q, & \text{ if } i\geq 2, 
\end{cases}
\]
where
\begin{align*}
s &= e_{vx_l}e_{x_lx_{l-1}}\dots e_{x_2x_{1}}e_{x_1y}, \\
s_1 &= e_{u_1x_1}e_{x_1x_2}\dots e_{x_{l-1}x_{l}}e_{x_lv}, \\
s_j &= e_{u_{j}u_{j-1}}\dots e_{u_2u_1} e_{u_1x_1} e_{x_{1}x_{2}}\dots e_{x_{l-j+1}x_{l-j+2}}, & (2 &\leq j\leq m), \\
p &= e_{yx_1}e_{x_1u_{1}}e_{u_1u_{2}}\dots e_{u_{i-1}u_{i}}, \\
t &= e_{x_{l-i+2}x_{l-i+1}}\dots e_{x_2x_{1}}e_{x_1y}, \\
t_j &= e_{x_{l-j+2}x_{l-j+1}}\dots e_{x_2x_{1}}e_{x_1u_1}e_{u_1u_2}\dots e_{u_{j-1}u_j}, & (2 &\leq j\leq m), \\
t_1 &= e_{vx_{l}}e_{x_lx_{l-1}}\dots e_{x_2x_{1}}e_{x_1u_1}, \\
q &= e_{yx_1}e_{x_1x_2}\dots e_{x_{l-1}{x_l}}e_{x_lv}. 
\end{align*}
Then $r$ transposes $u_i$ and $v$ and fixes all other vertices of $\Gamma$ outside the defect set. 
\end{proof}

Note 
that Lemma~\ref{lem:sgraph2} is going to be useful whenever $\Gamma$ contains a node with degree at least 3. 

\begin{lem}\label{lem:kcycleplusvertex}
Let $k \geq 2$, 
$\Gamma' = \left( V', E' \right)$ be such that $\abs{V'} > k$ and its defect $k$ group is transitive
(e.g.\ if $\Gamma'$ is a cycle with at least $k+1$ vertices). 
Let $\Gamma=\left( V'\cup \halmaz{ v }, E' \cup \halmaz{x_1 v} \right)$ for a new vertex $v$ and some $x_1 \in \Gamma'$, 
where the degree of $x_1$ in $\Gamma'$ is at least 2. 
Then the defect $k$ group of $\Gamma$ is isomorphic to  $S_{n-k}$.
\end{lem}

\begin{proof}
Let $n$ be the number of vertices of $\Gamma$, 
then $n \geq k+2$. 
Let the vertices of $\Gamma'$ be $y, x_1, x_2, \dots , x_{k-1}, u_{1}, u_{2}, \dots , u_{n-k-1}$ such that $u_1$ and $y$ are neighbors of $x_1$ in $\Gamma'$. 
Let the defect set be $\halmaz{y, x_1, \dots , x_{k-1}}$. 
Applying Lemma~\ref{lem:sgraph2} to the subgraph with vertices $\halmaz{x_1, v, y, u_1}$
we obtain that the defect $k$ group of $\Gamma$ contains the transposition $(u_1, v)$. 
Since  the defect $k$ group
of $\Gamma'$ is transitive and contained in the defect $k$ group of $\Gamma$ by Lemma~\ref{lem:simple}, 
the defect $k$ group of $\Gamma$ contains the transposition $(u_i, v)$ for all $1\leq i \leq n-k-1$. 
Therefore, the defect $k$ group of $\Gamma$ is isomorphic to $S_{n-k}$.
\end{proof}

Motivated by Lemma~\ref{lem:kcycleplusvertex}, 
we define the \emph{$k$-sub\-graphs} and the \emph{maximal $k$-subgraphs} of a graph $\Gamma$. 

\begin{dfn}
Let $\Gamma$ be a simple connected graph, $k\geq 2$.
A connected subgraph $\Gamma' \subseteq \Gamma$ is called a \emph{$k$-subgraph}
if its defect $k$ group is the symmetric group of degree $\abs{\Gamma'}-k$. 
A $k$-subgraph is a \emph{maximal $k$-subgraph} if it has no proper extension in $\Gamma$ to a $k$-subgraph.
Finally, 
we say that a $k$-subgraph $\Gamma'$ is \emph{nontrivial} if it contains a vertex having at least 3 distinct neighbors in $\Gamma'$. 
\end{dfn}
Note that every 
maximal $k$-subgraph is an induced subgraph. 
A trivial $k$-subgraph is either a line on $k+1$ points or a cycle on $k+1$ or $k+2$ points. 
Furthermore, 
a trivial maximal $k$-subgraph cannot be a cycle by Lemma~\ref{lem:kcycleplusvertex}, 
unless the graph itself is a cycle. 
Finally, 
any connected subgraph of $k+1$ points is trivially a $k$-subgraph, 
thus every connected subgraph of $k+1$ points is contained in a maximal $k$-subgraph.
Note 
that the intersection of two maximal $k$-subgraphs cannot contain more than $k$ vertices: 

\begin{lem}\label{lem:kintersection}
Let $\Gamma_1, \Gamma_2$ be $k$-subgraphs such that $\abs{ \Gamma_1\cap\Gamma_2 } > k$. 
Then $\Gamma_1\cup \Gamma_2$ is a $k$-subgraph, 
as well. 
\end{lem}

\begin{proof}
Choose the defect set $V_k$ such that $V_k \subsetneqq \Gamma_1 \cap \Gamma_2$, 
and let $v \in \left( \Gamma_1 \cap \Gamma_2 \right) \setminus V_k$. 
Then the symmetric groups acting on 
$\Gamma_1\setminus V_k$ and $\Gamma_2\setminus V_k$ 
are subgroups in the defect $k$ group of $\Gamma_1\cup\Gamma_2$. 
Thus, 
we can transpose every member of $\Gamma_i \setminus \left( V_k \cup \halmaz{v} \right)$ with $v$. 
Therefore, the defect $k$ group of $\Gamma_1 \cup \Gamma_2$ is the symmetric group on 
$\left( \Gamma_1\cup\Gamma_2 \right) \setminus V_k$. 
\end{proof}

\begin{lem}\label{lem:deg3}
Let $\Gamma$ be a simple connected graph, 
and let $\Gamma'$ be a $k$-subgraph of $\Gamma$. 
Let $x_1\in\Gamma'$, $v \notin \Gamma'$, 
and let $P = \left( x_1, x_2, \dots, x_l, v \right)$ be a shortest path between $x_1$ and $v$ in $\Gamma$ for some $l\leq k-1$. 
Assume that $x_1$ has at least 2 neighbors in $\Gamma'$ apart from $x_2$. 
Then the subgraph $\Gamma' \cup P$ is a $k$-subgraph. 
\end{lem}

\begin{proof}
First, consider the case $x_2, \dots , x_l \in \Gamma'$. 
Let $u, y$ be two neighbors of $x_1$ in $\Gamma'$ distinct from $x_2$, 
and choose the defect set $V_k$ such that it contains $y, x_1, \dots , x_l$ and does not contain $u$. 
By Lemma~\ref{lem:sgraph2} the defect $k$ group of $\Gamma'\cup \halmaz{ v }$ contains the transposition $(u, v)$.
Furthermore, 
the defect $k$ group of $\Gamma'$ is the whole symmetric group on $\Gamma' \setminus V_k$. 
Thus, 
the defect $k$ group of $\Gamma' \cup \halmaz{v}$ is the whole symmetric group on $\left( \Gamma' \setminus V_k \right) \cup \halmaz{v}$.

Now, 
if not all of $x_2, \dots , x_l$ are in $\Gamma'$, 
then, by the previous argument, one can add them (and then $v$) to $\Gamma'$ one by one, 
and obtain an increasing chain of $k$-subgraphs.
\end{proof}

As a corollary, 
we obtain that every vertex of degree at least 3 together with at least two of its neighbors is contained in exactly one nontrivial maximal $k$-subgraph. 

\begin{cor}\label{cor:kdeg3}
Let $\Gamma$ be a simple connected graph with $n$ vertices such that $n>k$,
and let $x_1$ be a vertex having degree at least $3$.
Then there exists exactly one maximal $k$-subgraph $\Gamma'$ containing $x_1$ such that $x_1$ has degree at least 2 in $\Gamma'$. 
Furthermore, 
$\Gamma'$ is a nontrivial $k$-subgraph, 
and if $\Gamma_{x_1}$ is the induced subgraph of the vertices in $\Gamma$ that are of at most distance $k-1$ from $x_1$, 
then $\Gamma_{x_1} \subseteq \Gamma'$. 
\end{cor}

\begin{proof}
Any connected subgraph of $\Gamma$ with $k+1$ vertices containing $x_1$ and any two of its neighbors is a $k$-subgraph. 
Thus there exists at least one maximal $k$-subgraph containing $x_1$ and two of its neighbors.

Let $\Gamma'$ be a maximal $k$-subgraph containing $x_1$ and at least two of its neighbors. 
Assume that $\Gamma_{x_1} \not\subseteq \Gamma'$. 
Let $v \in \Gamma_{x_1} \setminus \Gamma'$ be any vertex at a minimal distance from $x_1$, 
and let $P = (x_1, \dots , x_l, v)$ be a shortest path between $x_1$ and $v$. 
If $l = 1$, 
then $P=(x_1, v)$. 
Now $x_1$ has at least two neighbors in $\Gamma'$ apart from $v$, 
therefore $\Gamma' \cup P$ is a $k$-subgraph by Lemma~\ref{lem:deg3}, 
which contradicts the maximality of $\Gamma'$. 
Thus $l\geq 2$, 
in particular all neighbors of $x_1$ in $\Gamma$ are in $\Gamma'$, as well, 
and thus $\Gamma'$ is a nontrivial $k$-subgraph. 
Hence $x_1$ has at least two neighbors in $\Gamma'$ apart from $x_2$, 
therefore $\Gamma' \cup P$ is a $k$-subgraph by Lemma~\ref{lem:deg3}, 
which contradicts the maximality of $\Gamma'$. 
Thus $\Gamma_{x_1} \subseteq \Gamma'$. 

Now, 
assume that $\Gamma'$ and $\Gamma''$ are maximal $k$-subgraphs containing $x_1$ and at least two of its neighbors. 
Then $\Gamma_{x_1} \subseteq \Gamma'$ and $\Gamma_{x_1} \subseteq \Gamma''$. 
Note 
that either $\Gamma_{x_1} = \Gamma$ (and hence $\abs{\Gamma_{x_1}} = n >k$), 
or there exists a vertex $v \in \Gamma$ which is of distance exactly $k$ from $x_1$. 
Let $P = (x_1, \dots , x_k, v)$ be a shortest path between $x_1$ and $v$, 
and let $u$ and $y$ be two neighbors of $x_1$ distinct from $x_2$. 
Then $\halmaz{x_1, \dots , x_k , y, u} \subseteq \Gamma_{x_1}$, 
thus $\abs{\Gamma_{x_1}} > k$. 
Therefore $\abs{\Gamma' \cap \Gamma''} \geq \abs{\Gamma_{x_1}} > k$, 
yielding $\Gamma' = \Gamma''$ by Lemma~\ref{lem:kintersection}. 
\end{proof}

\begin{lem}\label{lem:kear}
Let $\Gamma'$ be a nontrivial $k$-subgraph of $\Gamma$, 
and let $P$ be a $\Gamma'$-ear. 
Then $\Gamma' \cup P$ is a (nontrivial) $k$-subgraph of $\Gamma$. 
\end{lem}

\begin{proof}
Let $\Gamma$, $\Gamma'$ and $P=\left( w_0,w_1,\dots w_i, w_{i+1} \right)$ be a counterexample, 
where $i$ is minimal. 
There exists a shortest path
$(w_0, y_1, \dots, y_l, w_{i+1})$  in $\Gamma'$ among those
where the degree of some $y_j$ or of $w_0$ or of $w_{i+1}$ is at least  $3$ in $\Gamma'$. 
(At least one such path exists, 
because $\Gamma'$ is connected, and is a nontrivial $k$-subgraph, 
hence contains a vertex of degree at least 3.)
For easier notation, 
let $y_0 = w_0$, $y_{l+1} = w_{i+1}$. 
Let $y' \in \Gamma' \setminus \halmaz{y_0, y_1, \dots, y_l, y_{l+1}}$ be a neighbor of $y_j$; 
this exists, because the degree of $y_j$ is at least 3, 
and otherwise a shorter path would exist between $w_0$ and $w_{i+1}$. 

If $j+1 \leq k-1$ (that is $j \leq k-2$), 
then by Lemma~\ref{lem:deg3} the induced subgraph on $\Gamma' \cup \halmaz{w_1}$ is a $k$-subgraph, 
thus $\Gamma' \cup \halmaz{w_1}$ with the ear $( w_1, \dots , w_i, w_{i+1} )$ is a counterexample with a shorter ear. 

Similarly, 
if $l-j+2 \leq k-1$ (that is $l+3-k \leq j$), 
then by Lemma~\ref{lem:deg3} the induced subgraph on $\Gamma' \cup \halmaz{w_i}$ is a $k$-subgraph, 
thus $\Gamma' \cup \halmaz{w_i}$ with the ear $( w_0, w_1, \dots , w_i )$ is a counterexample with a shorter ear. 

Finally, 
if $k-1 \leq j \leq l+2-k$, 
then $ 2k-3 \leq l$. 
Let $\Gamma''$ be the cycle $P \cup \left( y_0, y_1, \dots , y_l, y_{l+1} \right)$ together with $y'$ and the edge $y_j y'$. 
Then $\Gamma''$ is a $k$-subgraph by Lemma~\ref{lem:kcycleplusvertex}, 
$\abs{\Gamma' \cap \Gamma''} = l+2 \geq 2k-1 > k$, 
hence $\Gamma' \cup \Gamma'' = \Gamma' \cup P$ is a $k$-subgraph by Lemma~\ref{lem:kintersection}. 
\end{proof}

\begin{cor}\label{cor:2edge}
Let $\Gamma$ be a simple connected graph with $n$ vertices such that $n>k$,
and assume that $\Gamma$ is not a cycle. 
Suppose $uv$ is an edge contained in a cycle of $\Gamma$. 
Then there exists exactly one maximal $k$-subgraph $\Gamma'$ containing the edge $uv$. 
Furthermore, 
$\Gamma'$ is a nontrivial $k$-subgraph, 
and if $\Gamma_{uv}$ is the 2-edge connected component containing $uv$, 
then $\Gamma_{uv} \subseteq \Gamma'$. 
\end{cor}

\begin{proof}
Any connected subgraph of $\Gamma$ with $k+1$ vertices containing the edge $uv$ is a $k$-subgraph. 
Thus there exists at least one maximal $k$-subgraph $\Gamma'$ containing the edge $uv$. 
We prove first that $\Gamma'$ is a nontrivial $k$-subgraph, 
then prove $\Gamma_{uv} \subseteq \Gamma'$,
and only after that do we prove that $\Gamma'$ is unique. 

Assume first that $\Gamma'$ is a trivial $k$-subgraph. 
If $\Gamma'$ were a cycle, 
then $\Gamma \setminus \Gamma'$ contains at least  one vertex, 
because $\Gamma'$ is an induced subgraph of $\Gamma$.
Then Lemma~\ref{lem:kcycleplusvertex} contradicts the maximality of $\Gamma'$. 
Thus $\Gamma'$ is a line of $k+1$ vertices. 
Let $\Gamma_2$ be a shortest cycle containing $uv$. 
Now, 
there must exist a vertex in $\Gamma \setminus \Gamma_2$, 
otherwise either $\Gamma = \Gamma_2$ would be a cycle, 
or there would exist an edge in $\Gamma \setminus \Gamma_2$ yielding a shorter cycle than $\Gamma_2$ containing the edge $uv$. 
Let $x_2 \in \Gamma \setminus \Gamma_2$ be a neighbor of a vertex in $\Gamma_2$. 
By Lemma~\ref{lem:kcycleplusvertex} the induced subgraph on $\Gamma_2 \cup \halmaz{x_2}$ is a $k$-subgraph. 
Thus $\Gamma' \not\subseteq \Gamma_2$, 
otherwise $\Gamma'$ would not be a maximal $k$-subgraph. 
Let $x_1 \in \Gamma' \cap \Gamma_2$ be a vertex such that two of its neighbors are in $\Gamma_2$ and its third neighbor is some $x_2 \in \Gamma' \setminus \Gamma_2$. 
Note 
that every vertex in $\Gamma'$ is of distance at most $k-1$ from $x_1$, 
because $u,v \in \Gamma' \cap \Gamma_2$. 
Thus, 
if $\abs{\Gamma_2} \geq k+1$, 
then $\Gamma_2$ together with $x_2$ and the edge $x_1x_2$ is a $k$-subgraph by Lemma~\ref{lem:kcycleplusvertex}, 
and hence $\Gamma_2 \cup \Gamma'$ is a $k$-subgraph by Lemma~\ref{lem:deg3}, 
contradicting the maximality of $\Gamma'$. 
Otherwise, 
if $\abs{\Gamma_2} \leq k$, 
then every vertex in $\Gamma_2$ is of distance at most $k-1$ from $x_1$, 
and hence $\Gamma_2 \cup \Gamma'$ is a $k$-subgraph by Lemma~\ref{lem:deg3}, 
contradicting the maximality of $\Gamma'$. 
Therefore $\Gamma'$ is a nontrivial $k$-subgraph. 

Now we show that the two-edge connected component $\Gamma_{uv} \subseteq \Gamma'$. 
Let $\Gamma, \Gamma'$ be a counterexample to this such that the number of vertices of $\Gamma_{uv}$ is minimal, 
and among these counterexamples choose one where the number of edges of $\Gamma_{uv}$ is minimal. 
Using an ear-decomposition \cite{Robbins1939}, 
$\Gamma_{uv}$ is either a cycle, 
or there exists a 2-edge connected subgraph $\Gamma_1 \subseteq \Gamma_{uv}$
and there exists 
\begin{enumerate}
\item\label{it:ear}
either a $\Gamma_1$-ear $P$ such that $\Gamma_{uv} = \Gamma_1 \cup P$,
\item\label{it:cycle}
or a cycle $\Gamma_2$ such that $\abs{\Gamma_1 \cap \Gamma_2} = 1$ and $\Gamma_{uv} = \Gamma_1 \cup \Gamma_2$. 
\end{enumerate}
If $\Gamma_{uv}$ is a cycle containing the edge $uv$, 
and $\Gamma_{uv} \not \subseteq \Gamma'$, 
then going along the edges of $\Gamma_{uv}$, 
one can find a $\Gamma'$-ear $P \subseteq \Gamma_{uv}$. 
Then $\Gamma' \cup P$ is a $k$-subgraph by Lemma~\ref{lem:kear}, 
contradicting the maximality of $\Gamma'$. 
Thus $\Gamma_{uv}$ is not a cycle. 
Let us choose $\Gamma_1$ from cases~(\ref{it:ear})~and~(\ref{it:cycle}) so that it would have the least number of vertices. 

Assume first that case~(\ref{it:ear}) holds. 
By minimality of the counterexample, 
$\Gamma_1 \subseteq \Gamma'$. 
If $P \not \subseteq \Gamma'$, 
then going along the edges of $P$ one can find a $\Gamma'$-ear $P' \subseteq P$. 
But then $\Gamma' \cup P'$ is a $k$-subgraph by Lemma~\ref{lem:kear}, 
contradicting the maximality of $\Gamma'$. 

Assume now that case~(\ref{it:cycle}) holds. 
Again, by induction, 
$\Gamma_1 \subseteq \Gamma'$. 
If $\Gamma_2 \not \subseteq \Gamma'$, 
then either $\abs{\Gamma' \cap \Gamma_2} = 1$ or going along the edges of $\Gamma_2$ one can find a $\Gamma'$-ear $P' \subseteq \Gamma_2$. 
The latter case cannot happen, 
because then $\Gamma' \cup P'$ is a $k$-subgraph by Lemma~\ref{lem:kear}, 
contradicting the maximality of $\Gamma'$. 
Thus $\abs{\Gamma' \cap \Gamma_2} = 1$, 
and hence $\Gamma' \cap \Gamma_2 = \Gamma_1 \cap \Gamma_2$.  
Let $\Gamma_1 \cap \Gamma_2 = \halmaz{x_1}$, 
and let $v_1$ be a neighbor of $x_1$ in $\Gamma_1 \setminus  \Gamma_2$, 
and let $v_2$ be a neighbor of $x_1$ in $\Gamma_2 \setminus \Gamma_1$. 
If $\abs{\Gamma_2} \leq k$, 
then $\Gamma_2$ can be extended to a connected subgraph of $\Gamma$ having exactly $k+1$ vertices, 
which is a $k$-subgraph. 
If $\abs{\Gamma_2} \geq k+1$, 
then $\Gamma_2 \cup \halmaz{v_1}$ is a $k$-subgraph by Lemma~\ref{lem:kcycleplusvertex}. 
In any case, 
there exists a maximal $k$-subgraph $\Gamma_2' \supseteq \Gamma_2$. 
For notational convenience, 
let $\Gamma_1'$ denote the maximal $k$-subgraph $\Gamma'$ containing $\Gamma_1$. 
We prove that $\Gamma_2' = \Gamma_1'=\Gamma'$, 
thus $\Gamma'$ contains $\Gamma_2$, 
contradicting that we chose a counterexample. 

Now, 
both $\Gamma_1$ and $\Gamma_2$ contain at least two neighbors of $x_1$. 
Let $V_i \subseteq \Gamma_i$ be the set of vertices with distance at most $k-1$ from $x_1$ ($i \in \halmaz{1,2}$).  
If $\abs{\Gamma_i} \leq k$, 
then $V_i$ contains all vertices of $\Gamma_i$, 
otherwise $\abs{V_i} \geq k$ ($i \in \halmaz{1,2}$). 
By Lemma~\ref{lem:deg3}, 
the induced subgraph on $V_1$ is contained in $\Gamma_2'$. 
Thus, 
if $V_1$ contains all vertices of $\Gamma_1$, 
then $\Gamma_1 \subseteq \Gamma_2'$, 
hence we have $\Gamma_1' = \Gamma_2'$. 
Similarly, 
the induced subgraph on $V_2$ is contained in $\Gamma_1'$. 
Thus, 
if $V_2$ contains all vertices of $\Gamma_2$, 
then $\Gamma_2 \subseteq \Gamma_1'$, 
hence we have $\Gamma_1' = \Gamma_2'$. 
Otherwise, 
$\abs{\Gamma_1' \cap \Gamma_2'} \geq \abs{V_1} + \abs{V_2} - \abs{\halmaz{x_1}} \geq 2k-1 > k$, 
hence by Lemma~\ref{lem:kintersection} we have $\Gamma_1' = \Gamma_2'$. 

Finally, 
we prove uniqueness. 
Let $\Gamma'$ and $\Gamma''$ be two maximal $k$-subgraphs containing the edge $uv$. 
Then both $\Gamma'$ and $\Gamma''$ contain $\Gamma_{uv}$. 
If $\Gamma = \Gamma_{uv}$, 
then $\Gamma' = \Gamma_{uv} = \Gamma''$. 
Otherwise, 
there exists a vertex $x_2 \in \Gamma \setminus \Gamma_{uv}$ such that it has a neighbor $x_1 \in \Gamma_{uv}$.
Note 
that $x_1$ has degree at least 3 in $\Gamma$. 
Let $V_1$ be the vertices of $\Gamma$ of distance at most $k-1$ from $x_1$. 
Note 
that if $V_1$ does not contain all vertices of $\Gamma$, 
then $\abs{V_1} > k$. 
By 2-edge connectivity,
$\Gamma_{uv} \subseteq \Gamma'$ contains at least two neighbors of $x_1$, 
thus $V_1 \subseteq \Gamma'$ by Lemma~\ref{lem:deg3}. 
Similarly, 
$\Gamma_{uv} \subseteq \Gamma''$ contains at least two neighbors of $x_1$, 
thus $V_1 \subseteq \Gamma''$ by Lemma~\ref{lem:deg3}. 
If $V_1$ contains all vertices of $\Gamma$, 
then $\Gamma' = \Gamma = \Gamma''$. 
Otherwise, 
$\abs{\Gamma' \cap \Gamma''} \geq \abs{V_1} > k$, 
and $\Gamma' = \Gamma''$ by Lemma~\ref{lem:kintersection}. 
\end{proof}

Recall
that by \cite{Robbins1939} 
a strongly connected antisymmetric digraph becomes a 2-edge connected graph after forgetting the directions. 
Thus Rhodes's conjecture about strongly connected, 
antisymmetric digraphs \cite[Conjecture~6.51i~(3)--(4)]{wildbook}
follows immediately from the following theorem on 2-edge connected graphs: 

\begin{thm}\label{thm:defectk}
Let $n>k \geq 2$, 
$\Gamma$ be a $2$-edge connected simple graph having $n$ vertices. 
If $\Gamma$ is a cycle, 
then the defect $k$ group is $Z_{n-k}$. 
If $\Gamma$ is not a cycle, 
then the defect $k$ group is $S_{n-k}$. 
\end{thm}

\begin{proof}
If $\Gamma$ is a cycle, 
then its defect $k$ group is $Z_{n-k}$ by Lemma~\ref{lem:cycle}. 
Since $\Gamma$ is 2-edge connected with at least 3 vertices, 
every edge of $\Gamma$ is contained in a cycle. 
Thus, if $\Gamma$ is not a cycle, 
then the defect $k$ group is $S_{n-k}$ by Corollary~\ref{cor:2edge}. 
\end{proof}

The final part of this section is devoted to prove item~(\ref{it:defectkmain}) of Theorem~\ref{thm:main}. 
First, 
we define bridges in $\Gamma$: 

\begin{dfn}
A path $\left( x_1, \dots,  x_l \right)$ in a connected graph $\Gamma$ for some $l \geq 2$ is called a \textit{bridge} 
if the degree of $x_i$ in $\Gamma$ is $2$ for all $2 \leq i\leq l-1$, 
and if $\Gamma \setminus \halmaz{x_jx_{j+1}}$ is disconnected for all $1 \leq j\leq l-1$. 
The \emph{length} of the bridge $\left( x_1, \dots,  x_l \right)$ is $l$. 
\end{dfn}

The intersection of maximal $k$-subgraphs turn out to be bridges: 

\begin{lem}\label{lem:kcomp3}
Let $\Gamma_1$ and $\Gamma_2$ be distinct maximal $k$-subgraphs of the connected simple graph $\Gamma$. 
Assume that $\Gamma$ is not a cycle. 
Then $\Gamma_1 \cap \Gamma_2$ is either empty, 
or is a bridge $(x_1, \dots , x_l)$ such that 
\begin{enumerate}
\item\label{it:lk}
$l \leq k$, and
\item\label{it:x_1Gamma_1}
if $l\geq 2$ and $\Gamma_i \setminus \halmaz{x_1, \dots, x_l}$ ($i \in \halmaz{1,2}$) contains a neighbor of $x_1$ (resp.\ $x_l$), 
then $\Gamma_i$ contains all neighbors of $x_1$ (resp.\ $x_l$), 
\end{enumerate}
\end{lem}

\begin{proof} 
Note 
that $\Gamma_1$ and $\Gamma_2$ are induced subgraphs of $\Gamma$, 
thus so is $\Gamma_1 \cap \Gamma_2$. 

We prove first that $\Gamma_1 \cap \Gamma_2$ is connected (or empty) if $\Gamma_1$ is a nontrivial maximal $k$-subgraph. 
Suppose that $u, v \in \Gamma_1 \cap \Gamma_2$ are in
different components of $\Gamma_1\cap\Gamma_2$ such that the distance 
between $u$ and $v$ is minimal in $\Gamma_2$.
Due to the minimality, 
there exists a path $(u, x_1, \dots , x_l, v)$ such that 
$x_1, \dots, x_l \in \Gamma_2 \setminus \Gamma_1$. 
Then $P = (u, x_1, \dots,  x_l, v)$ is a $\Gamma_1$-ear,
and $\Gamma_1 \cup P$ would be a $k$-subgraph by Lemma~\ref{lem:kear}, 
contradicting the maximality of $\Gamma_1$.
Thus $\Gamma_1 \cap \Gamma_2$ is connected. 
One can prove similarly that $\Gamma_1 \cap \Gamma_2$ is connected if $\Gamma_2$ is a nontrivial maximal $k$-subgraph. 

Now we prove that $\Gamma_1 \cap \Gamma_2$ is connected, 
even if both $\Gamma_1$ and $\Gamma_2$ are trivial maximal $k$-subgraphs. 
As $\Gamma_1 \subsetneqq \Gamma$, $\Gamma_1$ cannot be a cycle hence must
be
a line $(x_1, \ldots, x_{k+1})$.
Note 
that the degree of $x_i$ in $\Gamma$ for $2 \leq i \leq k$ must be 2, 
otherwise a nontrivial maximal $k$-subgraph would contain $x_i$, 
and thus also $\Gamma_1$ by Corollary~\ref{cor:kdeg3}. 
In particular, 
if $\Gamma_1 \cap \Gamma_2$ is not connected, 
then $x_1, x_{k+1} \in \Gamma_1 \cap \Gamma_2$, 
$x_i \notin \Gamma_1 \cap \Gamma_2$ for some $2 \leq i\leq k$, 
and $\Gamma_1 \cup \Gamma_2$ would be a cycle. 
However, 
by Corollary~\ref{cor:2edge}, 
the edge $x_1x_2$ is contained in a unique nontrivial maximal $k$-subgraph, 
contradicting that it is also contained in the trivial maximal $k$-subgraph $\Gamma_1$. 

Now, we prove (\ref{it:lk}). 
By Corollary~\ref{cor:2edge}, 
$\Gamma_1 \cap \Gamma_2$ cannot contain any edge $uv$ which is contained in a cycle. 
As $\Gamma_1 \cap \Gamma_2$ is connected, 
it must be a tree. 
However, 
$\Gamma_1 \cap \Gamma_2$ cannot contain any vertex of degree at least 3 in $\Gamma_1 \cap \Gamma_2$, 
otherwise that vertex would be contained in a unique maximal $k$-subgraph by Corollary~\ref{cor:kdeg3}. 
Thus $\Gamma_1 \cap \Gamma_2$ is a path $(x_1, \dots , x_l)$. 
Now, 
$l\leq k$ by Lemma~\ref{lem:kintersection}, 
proving (\ref{it:lk}). 
Note 
that if any $x_i$ ($2\leq i\leq l-1$) is of degree at least 3 in $\Gamma$, 
then $\halmaz{x_{i-1}, x_i, x_{i+1}}$ is contained in a unique maximal $k$-subgraph by Corollary~\ref{cor:kdeg3}, 
a contradiction.
For (\ref{it:x_1Gamma_1}) observe that at least two neighbors of $x_1$ (resp.\ $x_l$) are in $\Gamma_i$, 
and thus all its neighbors must be in $\Gamma_i$ by Corollary~\ref{cor:kdeg3}. 
Finally, 
if $l\geq 2$ then $\Gamma \setminus \halmaz{x_jx_{j+1}}$ is disconnected for all $1\leq j\leq l-1$ follows immediately from Corollary~\ref{cor:2edge} and the fact that any edge that is not contained in any cycle disconnects the graph $\Gamma$.
\end{proof}

Edges of short maximal bridges (having length at most $k-1$) are contained in a unique maximal $k$-subgraph: 

\begin{lem}\label{lem:longbridgeedge}
Let $\Gamma$ be a simple connected graph with $n$ vertices such that $n>k$, 
and let $uv$ be an edge which is not contained in any cycle. 
Let $(x_1, \dots , x_l)$ be a longest bridge containing the edge $uv$. 
If $l\leq k-1$, 
then $uv$ is contained in a unique maximal $k$-subgraph $\Gamma'$, 
and furthermore, 
$\Gamma'$ is a nontrivial $k$-subgraph. 
\end{lem}

\begin{proof}
As $uv$ is not part of any cycle in $\Gamma$, 
$uv$ is a bridge of length 2. 
Note 
that a longest bridge $(x_1, \dots , x_l)$ containing $uv$ is unique, 
because as long as the degree of at least one of the path's end vertices is 2 in $\Gamma$, 
the path can be extended in that direction. 
The obtained path is the unique longest bridge containing $uv$. 

Let $\Gamma'$ be a maximal $k$-subgraph containing $uv$, 
and assume $l\leq k-1$. 
Note 
that the distance of $x_1$ and $x_l$ is $l-1 \leq k-2$. 
As $\abs{\Gamma} \geq k+1$, 
at least one of $x_1$ and $x_l$ has degree at least 3 in $\Gamma$, 
say $x_1$. 
We distinguish two cases according to the degree of $x_l$. 

Assume first that $x_l$ is of degree 1. 
As $\Gamma'$ is a connected subgraph having at least $k+1$ vertices, 
$\Gamma'$ must contain $x_1$ and at least two of its neighbors. 
Then by Corollary~\ref{cor:kdeg3} it contains all vertices of $\Gamma$ of distance at most $k-1$ from $x_1$. 
In particular, 
$\Gamma'$ must contain the bridge $(x_1, \dots , x_l)$.
However, 
there is a unique (nontrivial) maximal $k$-subgraph $\Gamma_1'$ containing $x_1$ and two of its neighbors by Corollary~\ref{cor:kdeg3}, 
and thus $\Gamma' = \Gamma_1'$ is that unique maximal $k$-subgraph. 

Assume now that $x_l$ is of degree at least 3. 
As $\Gamma'$ is a connected subgraph having at least $k+1$ vertices, 
$\Gamma'$ must contain $x_1$ and at least two of its neighbors, 
or $x_l$ and at least two of its neighbors. 
If $\Gamma'$ contains $x_1$ and at least two of its neighbors, 
then by Corollary~\ref{cor:kdeg3} it contains all vertices of $\Gamma$ of distance at most $k-1$ from $x_1$. 
In particular, 
$\Gamma'$ must contain the bridge $(x_1, \dots , x_l)$ and all of the neighbors of $x_l$. 
Similarly, 
one can prove that if $\Gamma'$ contains $x_l$ and two of its neighbors, 
then it also contains the bridge $(x_1, \dots , x_l)$ and all of the neighbors of $x_1$. 
However, 
there is a unique (nontrivial) maximal $k$-subgraph $\Gamma_1'$ containing $x_1$ and two of its neighbors by Corollary~\ref{cor:kdeg3}, 
and also a unique (nontrivial) maximal $k$-subgraph $\Gamma_l'$ containing $x_l$ and two of its neighbors by Corollary~\ref{cor:kdeg3}. 
Therefore $\Gamma'$ must equal to both $\Gamma_1'$ and $\Gamma_l'$, 
and hence is unique. 
\end{proof}

In particular, 
in non-cycle graphs trivial maximal $k$-subgraphs or intersections of two different maximal $k$-subgraphs consist of edges that are contained in long bridges 
(having length at least $k$). 
The key observation in proving item~(\ref{it:defectkmain}) of Theorem~\ref{thm:main} is that a defect $k$ group cannot move a vertex across a bridge of length at least $k$: 

\begin{lem}\label{lem:bridge}
Let $2\leq k\leq l$,  
$\Gamma_1$ and $\Gamma_2$ be disjoint connected subgraphs of the connected graph $\Gamma$, 
and $\left( x_1, x_2, \dots,  x_l \right)$ be a bridge in $\Gamma$
such that
$x_1 \dots,  x_{l} \notin \Gamma_1 \cup \Gamma_2$, 
$x_1$ has only neighbors in $\Gamma_1$ (except for $x_2$), 
$x_l$ has only neighbors in $\Gamma_2$ (except for $x_{l-1}$). 
Assume $\Gamma$ has no more vertices than $\Gamma_1 \cup \Gamma_2 \cup \left( x_1, \dots, x_l \right)$. 
Let the defect set be $V_k = \halmaz{x_1, \dots , x_k}$.
Then for any $u \in \Gamma_1$ and $v \in \Gamma_2$ there does not exist any permutation in $G_{k,V_k}$ which moves $u$ to $v$.
\end{lem}

\begin{proof}
Let $S = S_{\Gamma}$. 
Assume that there exists $u \in \Gamma_1$, 
$v \in \Gamma_2$, 
and a transformation $g \in S$ of defect $V_k$ such that $g \restriction_{V \setminus V_k} \in G_{k,V_k}$ and $ug=v$. 
Let $s_0 \in G_{k,V_k}$ be the unique idempotent power of $g$, 
that is $s_0$ is a transformation of defect $V_k$ that acts as the identity on $\Gamma \setminus V_k$. 
Then there exists a series of elementary collapsings $e_1, \dots , e_m$ such that $g = e_1 \dots e_m$.
For every $1 \leq d \leq m$ let $s_d = s_0e_1 \dots e_d$. 
Now, $s_m = s_0 e_1 \dots e_m = s_0 g = gs_0 = g$. 
In particular, 
both $s_m$ and $s_0$ are of defect $k$, 
hence $s_d$ is of defect $k$ for all $1\leq d \leq m$. 
Consequently, 
$\abs{\Gamma_1 s_d} = \abs{\Gamma_1}$, 
$\abs{\Gamma_2 s_d} = \abs{\Gamma_2}$ and 
$\Gamma_1 s_d \cap \Gamma_2 s_d = \emptyset$ for all $1\leq d \leq m$. 

For an arbitrary $s \in S$, 
let 
\begin{align*}
i(s) &= 
\begin{cases}
0, & \text{if } \Gamma_1 s \subseteq \Gamma_1, \\
l+1, & \text{if } \Gamma_1 s \not\subseteq \Gamma_1 \cup \halmaz{x_1, \dots , x_l}, \\
\displaystyle{\min_{1\leq i \leq l}\halmaz{\Gamma_1 s \subseteq \Gamma_1 \cup \halmaz{x_1, \dots , x_i}}}, & \text{otherwise}.
\end{cases}
\intertext{Similarly, let}
j(s) &= 
\begin{cases}
l+1, & \text{if } \Gamma_2 s \subseteq \Gamma_2, \\
0, & \text{if } \Gamma_2 s \not\subseteq \Gamma_2 \cup \halmaz{x_1, \dots , x_l}, \\
\displaystyle{\max_{1\leq j \leq l}\halmaz{\Gamma_2 s \subseteq \Gamma_2 \cup \halmaz{x_j, \dots , x_l}}}, & \text{otherwise}.
\end{cases}
\end{align*}
Note 
that for arbitrary $s \in S$ and elementary collapsing $e$, 
we have $\abs{i(s)-i({se})} \leq 1$,
$\abs{j(s)-j({se})} \leq 1$. 
Furthermore, 
both $\abs{i(s_d)-i({s_de})} = 1$ and $\abs{j(s_d)-j({s_de})}=1$ cannot happen at the same time for any $1\leq d \leq m$, 
because that would contradict $\Gamma_1 s_d \cap \Gamma_2 s_d \neq \emptyset$. 

For $s_0$ we have $i({s_0}) = 0  < l+1 = j({s_0})$, 
for $s_m$ we have $i({s_m}) = l+1 \geq j({s_m})$. 
Let $1 \leq d \leq m$ be minimal such that $i({s_d}) \geq j({s_d})$. 
Then $i({s_{d-1}}) < j({s_{d-1}})$. 
From $s_{d-1}$ to $s_d$ either $i$ or $j$ can change and by at most 1, 
thus $i({s_d}) = j({s_d})$. 
If $i({s_d}) = j({s_d}) \in \halmaz{1, \dots , l}$, 
then $x_{i(s_d)} \in \Gamma_1 s_d \cap \Gamma_2 s_d$, 
contradicting $\Gamma_1 s_d \cap \Gamma_2 s_d = \emptyset$. 
Thus $i({s_d}) = j({s_d}) \notin \halmaz{1, \dots , l}$. 
Assume $i(s_d) = j(s_d) = l+1$, 
the case $i(s_d) = j(s_d) = 0$ can be handled similarly.

Now, 
$j(s_d) = l+1$ yields $\Gamma_2 s_d \subseteq \Gamma_2$. 
Furthermore, 
$\abs{\Gamma_2 s_d} = \abs{\Gamma_2}$, 
thus $\Gamma_2 s_d = \Gamma_2$. 
From $i(s_d) = l+1$ we have $\Gamma_1 s_d \cap \Gamma_2 \neq \emptyset$. 
Thus $\Gamma_1 s_d \cap \Gamma_2 s_d = \Gamma_1 s_d \cap \Gamma_2 \neq \emptyset$, 
a contradiction. 
\end{proof}

\begin{cor}\label{cor:dir}
Let $\Gamma_1$ and $\Gamma_2$ be connected subgraphs of $\Gamma$ such that $\Gamma_1 \cap \Gamma_2$ is a length $k$ bridge in $\Gamma$. 
Let $V_k = \Gamma_1 \cap \Gamma_2$ be the defect set. 
Let $G_i$ be the defect $k$ group of $\Gamma_i$, 
$G$ be the defect $k$ group of $\Gamma_1 \cup \Gamma_2$. 
Then 
\[
G = G_1 \times G_2. 
\]
\end{cor}

\begin{proof}
By Lemma~\ref{lem:simple} we have $G_1, G_2 \leq G$. 
Since $G_1$ and $G_2$ act on disjoint vertices, their elements commute.
Thus $G_1\times G_2 \leq G$. 
Now, $V_k$ is a bridge of length $k$, 
thus by Lemma~\ref{lem:bridge} 
(applied to the disjoint subgraphs $\Gamma_1 \setminus V_k$ and $\Gamma_2 \setminus V_k$) 
there exists no element of $G$ moving a vertex from $\Gamma_1$ to $\Gamma_2$ or vice versa. 
Therefore $G \leq G_1\times G_2$.
\end{proof}

Finally, 
we are ready to prove item~(\ref{it:defectkmain}) of Theorem~\ref{thm:main}. 

\begin{proof}[Proof of item~(\ref{it:defectkmain}) of Theorem~\ref{thm:main}.]
If $\Gamma$ is a cycle, 
then its defect $k$ group is $Z_{n-k}$ by Lemma~\ref{lem:cycle}. 
Otherwise, 
we prove the theorem by induction on the number of maximal $k$-subgraphs of $\Gamma$. 
If $\Gamma$ is a maximal $k$-subgraph, 
then the theorem holds,
and the defect $k$ group of $\Gamma$ is $S_{n-k}$. 
In the following we assume that $\Gamma$ contains $m$-many maximal $k$-subgraphs for some $m\geq 2$, 
and that the theorem holds for all graphs with at most $(m-1)$-many maximal $k$-subgraphs. 

We consider two cases. 
Assume first that there exists a degree 1 vertex $x_1 \in \Gamma$, 
such that there exists a path $(x_1, \dots, x_{k+1})$ which is a bridge. 
Let $\Gamma_1$ be the path $(x_1, \dots, x_{k+1})$, 
and let $\Gamma_2$ be $\Gamma \setminus \halmaz{x_1}$. 
Now, 
$\Gamma_1$ is a trivial maximal $k$-subgraph, 
hence $\Gamma_2$ contains the same maximal $k$-subgraphs as $\Gamma$ except $\Gamma_1$. 
Furthermore, 
$\Gamma_2$ is connected, 
and cannot be a cycle because the degree of $x_2$ in $\Gamma_2$ is 1. 
Let the sizes of the maximal $k$-subgraphs of $\Gamma_2$ be $n_2, \dots , n_m$, 
then by induction the defect $k$ group of $\Gamma_1$ is $S_{n_2-k} \times \dots \times S_{n_m-k}$. 
The size of $\Gamma_1$ is $n_1 = k+1$, 
its defect $k$-group is $S_{n_1-k}$. 
Furthermore, 
$\Gamma_1 \cap \Gamma_2$ is a bridge of length $k$. 
By Corollary~\ref{cor:dir} the defect $k$-group of $\Gamma$ is $S_{n_1-k} \times S_{n_2-k} \times \dots \times S_{n_m-k}$.

In the second case, no degree 1 vertex $x_1$ is in a path $(x_1, \dots , x_{k+1})$ which is a bridge. 
Then any maximal bridge $(x_1, \dots, x_l)$  with a degree 1 vertex  $x_1$ has length $l \leq k$, 
and, as the bridge cannot be extended, 
$x_l$ must have degree at least 3. 
Moreover, $(x_1, ... , x_l)$ lies in a maximal $k$-subgraph containing $x_l$ and all its neighbors by Lemma~\ref{lem:longbridgeedge}~and~Corollary~\ref{cor:kdeg3}.  
In particular every bridge in $\Gamma$ of length at least $k+1$ occurs between nodes of degree at least 3. 
Hence every bridge of length at least $k+1$ occurs between two nontrivial maximal $k$-subgraphs by Corollary~\ref{cor:kdeg3}. 
For every vertex $v$ having degree at least 3 in $\Gamma$, 
let $\Gamma_v$ be the unique maximal $k$-subgraph containing $v$ and all its neighbors (Corollary~\ref{cor:kdeg3}). 
By definition, 
these are all the nontrivial maximal $k$-subgraphs of $\Gamma$. 

Let $\Gamma^k$ be the graph whose vertices are the nontrivial maximal $k$-subgraphs, 
and $\Gamma_u\Gamma_v$ is an edge in $\Gamma^k$ (for $\Gamma_u \neq \Gamma_v$) if and only if 
there exists a bridge in $\Gamma$ between a vertex $u' \in \Gamma_u$ of degree at least 3 in $\Gamma_u$ 
and a vertex $v' \in \Gamma_v$ of degree at least 3 in $\Gamma_v$. 
By Corollary~\ref{cor:2edge}, 
$\Gamma_u = \Gamma_v$ if $u$ and $v$ are in the same 2-edge connected component. 
As the 2-edge connected components of $\Gamma$ form a tree, 
the graph $\Gamma^k$ is a tree. 

Now, $\Gamma^k$ has $m$ vertices. 
Let $\Gamma_1$ be a leaf in $\Gamma^k$, 
and let $\Gamma_m$ be its unique neighbor in $\Gamma^k$. 
Let $x_1 \in \Gamma_1$ and $x_l \in \Gamma_m$ be the unique vertices of degree at least 3 in $\Gamma_i$ ($i \in \halmaz{1, l}$) such that there exists a bridge $P = ( x_1, \dots, x_l)$ in $\Gamma$. 
Note that the length of $P$ is at least $k$, 
otherwise 
$\Gamma_1 = \Gamma_m$ would follow by Lemma~\ref{lem:longbridgeedge}. 
Furthermore, 
any other bridge having an endpoint in $\Gamma_1$ must be of length at most $k$, 
because every degree 1 vertex is of distance at most $k-1$ from a vertex of degree at least 3. 
Thus every bridge other than $P$ and having an endpoint in $\Gamma_1$ is a subset of $\Gamma_1$ by Corollary~\ref{cor:kdeg3}. 

Let $\Gamma_2 = \left( \Gamma \setminus \Gamma_1 \right) \cup P$. 
Now, 
$\Gamma_1$ is a maximal $k$-subgraph, 
$\Gamma_2$ has one less maximal $k$-subgraphs than $\Gamma$. 
Furthermore, 
$\Gamma_2$ is connected, 
because every bridge other than $P$ and having an endpoint in $\Gamma_1$ is a subset of $\Gamma_1$. 
Finally, 
$\Gamma_2$ is not a cycle, 
because it contains the vertex $x_1$ which is of degree 1 in $\Gamma_2$. 
Let the sizes of the maximal $k$-subgraphs of $\Gamma_2$ be $n_2, \dots , n_m$, 
then by induction the defect $k$ group of $\Gamma_1$ is $S_{n_2-k} \times \dots \times S_{n_m-k}$. 
Let the size of $\Gamma_1$ be $n_1$, 
its defect $k$-group is $S_{n_1-k}$. 
Furthermore, 
$\Gamma_1 \cap \Gamma_2$ is a bridge of length $k$. 
By Corollary~\ref{cor:dir} the defect $k$-group of $\Gamma$ is $S_{n_1-k} \times S_{n_2-k} \times \dots \times S_{n_m-k}$. 
\end{proof}

\section{An algorithm to calculate the defect $k$ group}\label{sec:algorithm}
Note 
that by items~(\ref{it:graph})~and~(\ref{it:2-vertex}) of Theorem~\ref{thm:main} the defect 1 group can be trivially computed in $O \left( \abs{E} \right)$ time by first determining the 2-vertex connected components \cite{Tarjan2vertex}, 
and 
whether each is a cycle, 
the exceptional graph (Figure~\ref{fig:exceptionalgraph})
or if not, whether or not it is bipartite.  

For $k \geq 2$ one can check first 
if $\Gamma$ is a cycle (and then the defect group is $Z_{n-k}$) or a path (and then the defect group is trivial). 
In the following, 
we give a linear algorithm (running in $O\left( \abs{E} \right)$ time) to determine the maximal $k$-subgraphs ($k \geq 2$) of a connected graph $\Gamma$ 
having $n$ vertices, $\abs{E}$ edges where at least one vertex is of degree at least 3.

During the algorithm we color the vertices. 
Let us call a maximal subgraph with vertices having the same color a \emph{monochromatic component}. 
First, 
one finds all 2-edge connected components and the tree of two-edge connected components in $O \left( \abs{E} \right)$ time using e.g.~\cite{Tarjan2edge}. 
Color the vertices of the nontrivial (i.e.\ having size greater than 1) 2-edge connected components such that two distinct vertices have the same color if and only if they are in the same nontrivial 2-edge connected component. 
Furthermore, 
color the uncolored vertices having degree at least 3 by different colors from each other and from the colors of the 2-edge connected components. 
Then the monochromatic components are each contained in a unique nontrivial maximal $k$-subgraph by Corollaries~\ref{cor:kdeg3}~and~\ref{cor:2edge} 
(a nontrivial maximal $k$-subgraph may contain more than one of these monochromatic components). 
Furthermore, 
the monochromatic components and the degree 1 vertices are connected by bridges. 
If any of the bridges connecting two monochromatic components is of length at most $k-1$, 
then recolor the two monochromatic components at the ends of the bridge and the vertices of the bridge by the same color, 
because these are contained in the same maximal $k$-subgraph by Corollary~\ref{cor:kdeg3}. 
Similarly, if any of the bridges connecting a monochromatic component and a degree 1 vertex is of length at most $k-1$, 
then recolor the monochromatic component and the vertices of the bridge by the same color, 
because these are contained in the same maximal $k$-subgraph by Lemma~\ref{lem:longbridgeedge}. 
Repeat recoloring along all bridges of length at most $k-1$ in $O \left( \abs{E} \right)$ time. 
Then we obtain monochromatic components $\Gamma_1, \dots , \Gamma_l$ connected by long bridges (i.e.\ bridges of length at least $k$), 
and possibly some long bridges to degree 1 vertices. 
Now, we have finished coloring. 

For every $1\leq i\leq l$,
 let $\Gamma_i'$ be the induced subgraph having all vertices of distance at most $k-1$ from $\Gamma_i$, 
which can be obtained in $O\left( \abs{E} \right)$ time by adding the appropriate $k-1$ vertices of the long bridges to the appropriate monochromatic component. 
Note 
that the obtained induced subgraphs are not necessarily disjoint. 
Then $\Gamma_1', \dots , \Gamma_l'$ are the nontrivial maximal $k$-subgraphs of $\Gamma$ by Lemma~\ref{lem:longbridgeedge}. 
Again, 
by Lemma~\ref{lem:longbridgeedge}, 
the trivial maximal $k$-subgraphs of $\Gamma$ are the paths containing exactly $k+1$ vertices in a long bridge. 
These can also be computed in $O \left( \abs{E} \right)$ time by going through all long bridges. 
By item~(\ref{it:defectkmain}) of Theorem~\ref{thm:main}, 
the defect $k$ group of $\Gamma$ as a permutation group is the direct product of the defect $k$ groups of $\Gamma_1', \dots \Gamma_l'$, 
and the defect $k$ groups of the trivial maximal $k$-subgraphs.

\section{Complexity of the flow semigroup of (di)graphs}\label{sec:cpx}

In this section we apply our results and the complexity lower bounds of \cite{ComplexityLowerBounds} 
to verify \cite[Conjecture~6.51i (1)]{wildbook} for 2-vertex connected graphs.
That is, 
we prove that the Krohn--Rhodes (or group-) complexity of the flow semigroup of a 2-vertex connected graph with $n$ vertices is $n-2$ (item~\ref{it:compl2vertex} of Theorem~\ref{thm:main}). 
Then we derive item~\ref{it:compl2edge} of Theorem~\ref{thm:main} as a further consequences of our results.

For standard definitions on wreath product of semigroups, 
we refer the reader to e.g.~\cite[Definition~2.2]{wildbook}. 
A finite semigroup $S$ is called \emph{combinatorial} if and only if every maximal subgroup of $S$ has one element. 
Recall that the \emph{Krohn--Rhodes (or group-) complexity of a finite semigroup $S$} (denoted by $\cpx{S}$) is the smallest non-negative integer $n$ such that $S$ is a homomorphic image of a subsemigroup of the iterated wreath product
\[
C_n \wr G_n \wr \dots \wr C_1 \wr G_1 \wr C_0, 
\]
where $G_1, \dots , G_n$ are finite groups, 
$C_0, \dots , C_n$ are finite combinatorial semigroups, 
and $\wr$ denotes the wreath product 
(for the precise definition, see e.g.~\cite[Definition~3.13]{wildbook}). 
The definition immediately implies that 
if a finite semigroup $S$ is the homomorphic image of a subsemigroup of $T$, 
then $\cpx{S} \leq \cpx{T}$. 
More can be found on the complexity of semigroups in e.g.~\cite[Chapter~3]{wildbook}. 
We need the following results on the complexity of semigroups. 

\begin{lem}[{\cite[Prop.~6.49(b)]{wildbook}}]\label{lem:K_n}
The flow semigroup $K_n$ of the complete graph on $n \geq 2$ vertices 
has $\cpx{K_n} = n-2$. 
\end{lem}

\begin{lem}[{\cite[Sec.~3.7]{ComplexityLowerBounds}}]\label{lem:F_n}
The complexity of the full transformation semigroup $F_n$ on $n$ points is $\cpx{F_n}=n-1$. 
\end{lem}

The well-known \emph{$\mathcal{L}$-order} is a pre-order, 
i.e.\  a transitive and reflexive binary relation,
on the elements of a semigroup $S$ given by 
$s_1 \succeq_{\mathcal{L}} s_2$ if $s_1 = s_2$ or $ss_1 = s_2$ for  some $s \in S$. 
The \emph{$\mathcal{L}$-classes}  of $S$ are the equivalence classes of the $\mathcal{L}$-order. 
The $\mathcal L$-classes are thus partially ordered by  $L_1 \succeq_{\mathcal{L}} L_2$ if and only if   $SL_1 \cup L_1 \supseteq SL_2 \cup L_2$.
One says that a finite semigroup $S$ is a \emph{$T_1$-semigroup} if it is generated by some $\succeq_{\mathcal{L}}$-chain of  its $\mathcal{L}$-classes, i.e.\ 
if  there exist $\mathcal L$-classes $L_1 \succeq_{\mathcal{L}} \dots \succeq_{\mathcal{L}} L_m$ of  $S$ such that $S = \langle  L_1 \cup \ldots \cup  L_m \rangle$.
Equivalently, $S$ is a $T_1$-semigroup if there exist $U_i \subseteq L_i $ ($1 \leq i \leq m$)  for such a chain of $\mathcal L$-classes of $S$ such that $S=  \langle U_1 \cup \ldots \cup U_m \rangle$.

\begin{lem}[{\cite[Lemma 3.5(b)]{ComplexityLowerBounds}}]\label{lem:LB}
Let $S$ be a noncombinatorial $T_1$-semigroup. 
Then
  \[
  \cpx{S} \geq 1 + \cpx{EG(S)},
\]
  where $EG(S)$ is the subsemigroup of $S$ generated by all its idempotents.
\end{lem}

Now we prove \cite[Conjecture~6.51i (1)]{wildbook} for 2-vertex connected graphs. 

\begin{proof}[Proof of item~\ref{it:compl2vertex} of Theorem~\ref{thm:main}]
Let $\Gamma$ be a 2-vertex connected simple graph with $n \geq 2$ vertices. 
Let $K_n$ denote the flow semigroup of the complete graph on vertices $V$, 
where $\abs{V} = n$. 
Then $\cpx{S_{\Gamma}} \leq \cpx{K_n} = n-2$ by Lemma~\ref{lem:K_n}.
We proceed by induction on $n$.
If $n \leq 3$, then $\Gamma$ is a complete graph, 
and $\cpx{S_{\Gamma}}=n-2$ by Lemma~\ref{lem:K_n}. 
From now on we assume $n>3$ and $\Gamma = (V, E)$. 

\textbf{Case 1.} Assume first that $\Gamma$ is not a cycle. 
Let $(u,v)$ and $(x,y)$ be two disjoint edges in $\Gamma$.
 Let $G_1$ be the defect~1 group with defect set $V\setminus\halmaz{u}$ and idempotent $e_{uv}$ as its identity element. 
 Then $e_{uv}  \succeq_{\mathcal{L}} e_{xy}e_{uv} = e_{uv}e_{xy}$.  
 Let $T$ be $\left\langle G_1 \cup \halmaz{e_{uv}e_{xy}} \right\rangle$. 
 Since $G_1 \succeq_{\mathcal{L}} \halmaz{e_{uv}e_{xy}}$ is an ${\mathcal{L}}$-chain  in  $T$,  
  $T$ is a $T_1$-semigroup.   
  Furthermore, $T$ is noncombinatorial since $G_1$ is nontrivial. 
  Thus, by Lemma~\ref{lem:LB}
    \begin{equation}\label{eq:LB1}
  \cpx{T} \geq 1 + \cpx{EG(T)}.
  \end{equation}
 
   Let $\Gamma'$ be the complete graph on $V\setminus\halmaz{u}$.
  Let $a,b\in V\setminus\halmaz{u}$ be arbitrary distinct vertices.   
   By item~(\ref{it:2-vertex}) of Theorem~\ref{thm:main}, $G_1$ is 2-transitive.  
   Let $\pi \in G_1$ be such that $\pi(x)=a$ and $\pi(y)=b$.
   There is a positive integer $\omega>1$, with $\pi^{\omega}=e_{uv}$.
In particular, 
$e_{uv}$ commutes with $\pi$. 
   Observe that 
\begin{align*}
\pi^{\omega-1} e_{uv}e_{xy} \pi = e_{uv} \left( \pi^{\omega-1} e_{xy} \pi \right) &= e_{uv} e_{ab}, \text{ and thus} \\
\left(\pi^{\omega-1} e_{xy}e_{uv} \pi\right)\restriction_{V\setminus\halmaz{u}} = e_{ab}.
\end{align*}
That is, 
we obtain the generators $e_{ab}$ of $S_{\Gamma'}$ by restricting the idempotents $e_{uv}  e_{ab} \in T$ to $V \setminus \halmaz{u}$. 
Therefore, 
$S_{\Gamma'}$ is a homomorphic image of a subsemigroup of $EG(T)$, 
yielding 
\[
\cpx{EG(T)}\geq \cpx{S_{\Gamma'}}.
\]
By induction,    $\cpx{S_{\Gamma'}}=n-3$.
Applying \eqref{eq:LB1}, we obtain $\cpx{T}\geq n-2$.
Since $T$ is a subsemigroup of $S_\Gamma$, 
we obtain $\cpx{S_\Gamma} \geq \cpx{T}\geq n-2$.
  
  \textbf{Case 2.} Assume now that $\Gamma$ is the $n$-node cycle $(u, v_1,\dots,v_{n-1})$.
   Then $(u,v_1)$ and $(v_2,v_3)$ are disjoint edges.
   Let $G_1 \simeq Z_{n-1}$ be the defect~1 group with defect set $V\setminus\halmaz{u}$ and idempotent $e_{uv_1}$ as its identity element. 
   Let $\pi$ be a generator of $G_1$ with cycle structure $\left( v_1, \dots , v_{n-1}\right)$.
 Then $e_{uv_1}  \succeq_{\mathcal{L}} e_{v_2v_3}e_{uv_1} = e_{uv_1} e_{v_2v_3}$.  
 Let $T$ be $\left\langle G_1 \cup \halmaz{e_{uv_1}e_{v_2v_3}} \right\rangle$. 
 Since $G_1 \succeq_{\mathcal{L}} \halmaz{e_{uv_1}e_{v_2v_3}}$ is an ${\mathcal{L}}$-chain  in  $T$,  
  $T$ is a $T_1$-semigroup.   
  Furthermore, $T$ is noncombinatorial since $G_1$ is nontrivial. 
   Thus, by Lemma~\ref{lem:LB}
    \begin{equation}\label{eq:LB2}
  \cpx{T} \geq 1 + \cpx{EG(T)}.
  \end{equation} 

Let $\Gamma'$ be an $(n-1)$-node cycle with nodes $V \setminus \halmaz{u} = \halmaz{v_1, \dots, v_{n-1}}$.
Note 
that $e_{uv_1} = \pi^{n-1}$, 
and therefore $e_{u v_1}$ commutes with $\pi$. 
  Let $v_{i-1}, v_i, v_{i+1} \in V\setminus\halmaz{u}$ be three neighboring nodes in $\Gamma'$, where the
  indices are in $\halmaz{1,\dots,n-1}$ taken modulo $n-1$.   
    Observe that
\begin{align*}
\pi^{n-2} e_{uv_1}  e_{v_{i-1}v_i }\pi = e_{uv_1}  \left(\pi^{n-2} e_{v_{i-1}v_i } \pi \right) &= e_{uv_1}  e_{v_i v_{i+1}}, \text{ and thus} \\
\left(\pi^{n-2} e_{uv_1}  e_{v_{i-1}v_i }\pi\right)\restriction_{V\setminus\halmaz{u}} &= e_{v_i v_{i+1}}.
\end{align*}
That is, 
we obtain the generators $e_{v_i v_{i+1}}$ of $S_{\Gamma'}$ by restricting the idempotents $e_{u v_1}  e_{v_i v_{i+1}} \in T$ to $V \setminus \halmaz{u}$. 
Therefore, 
$S_{\Gamma'}$ is a homomorphic image of a subsemigroup of $EG(T)$, 
yielding 
\[
\cpx{EG(T)}\geq \cpx{S_{\Gamma'}}.
\]
By induction,    $\cpx{S_{\Gamma'}}=n-3$.
Applying \eqref{eq:LB2}, we obtain $\cpx{T}\geq n-2$.
Since $T$ is a subsemigroup of $S_\Gamma$, 
we have  $\cpx{S_\Gamma} \geq \cpx{T}\geq n-2$.
 \end{proof}
 
Note 
that by Lemma~\ref{lem:reverseedge} a strongly connected digraph has the same flow semigroup as the corresponding graph. 
Thus, 
item~\ref{it:compl2vertex} of Theorem~\ref{thm:main} proves Rhodes's conjecture \cite[Conjecture~6.51i (1)]{wildbook} for 2-vertex connected strongly connected digraphs, 
as well. 
The following lemma bounds the complexity in the remaining cases. 

\begin{lem}
Let $k$ be the smallest positive integer such that for a graph $\Gamma$ the flow semigroup $S_{\Gamma}$ has defect $k$ group $S_{n-k}$.
Then $\cpx{S_{\Gamma}}\geq n-1-k$.
\end{lem}
\begin{proof}
Assume first $k=n-1$. Then the lemma holds trivially. 
From now on, assume $k\leq n-2$.  Let $uv$ be an edge in $\Gamma$. 
Let $V_k$ be an arbitrary $k$-element subset of the vertex set $V$ disjoint from $\halmaz{u,v}$.
Let $G_k$ be the defect $k$ group with defect set $V_k$.
Let $S$ be the subsemigroup of $S_{\Gamma}$ generated by  $G_k$ and $e_{uv}$.
As $G_k \simeq S_{n-k}$, we have that $S$ is the semigroup of all transformations on $V\setminus V_k$.
Hence, $\cpx{S}=\cpx{F_{n-k}}=n-k-1$ by Lemma~\ref{lem:F_n}.  
Whence,
$\cpx{S_{\Gamma}} \geq \cpx{S}= n-k-1$.
\end{proof}

By Theorem~\ref{thm:defectk}, 
it immediately follows that the complexity of the flow semigroup of a 2-edge connected graph $\Gamma$ is at least $n-3$.
Furthermore, 
$\cpx{S_{\Gamma}} \leq \cpx{K_n} = n-2$ by Lemma~\ref{lem:K_n}.
This finishes the proof of item~\ref{it:compl2edge} of Theorem~\ref{thm:main}.

\bibliographystyle{abbrv}
\bibliography{../Bib/graphsemi}

\end{document}